\documentclass[11pt]{amsart}

\usepackage{amsthm}
\usepackage{amsmath}
\usepackage{amssymb}
\usepackage{latexsym}
\usepackage{verbatim}
\usepackage[mathscr]{eucal}
\usepackage{oldgerm}
\usepackage{amscd}
\usepackage[all]{xy}
\usepackage{multirow}
\usepackage{hyperref}
\usepackage{adjustbox}
\usepackage[dvipfdmx]{xcolor}

\newlength{\margins}
\usepackage[top=\margins,bottom=\margins,left=\margins,right=\margins]{geometry}

\begin{document}

%Greek letters

\newcommand{\alp}{\alpha}
\newcommand{\bet}{\beta}
\newcommand{\gam}{\gamma}
\newcommand{\del}{\delta}
\newcommand{\eps}{\epsilon}
\newcommand{\zet}{\zeta}
\newcommand{\tht}{\theta}
\newcommand{\iot}{\iota}
\newcommand{\kap}{\kappa}
\newcommand{\lam}{\lambda}
\newcommand{\sig}{\sigma}
\newcommand{\ups}{\upsilon}
\newcommand{\ome}{\omega}
\newcommand{\vep}{\varepsilon}
\newcommand{\vth}{\vartheta}
\newcommand{\vpi}{\varpi}
\newcommand{\vrh}{\varrho}
\newcommand{\vsi}{\varsigma}
\newcommand{\vph}{\varphi}
\newcommand{\Gam}{\Gamma}
\newcommand{\Del}{\Delta}
\newcommand{\Tht}{\Theta}
\newcommand{\Lam}{\Lambda}
\newcommand{\Sig}{\Sigma}
\newcommand{\Ups}{\Upsilon}
\newcommand{\Ome}{\Omega}

%fraktur letters

\newcommand{\frka}{{\mathfrak a}}    \newcommand{\frkA}{{\mathfrak A}}
\newcommand{\frkb}{{\mathfrak b}}    \newcommand{\frkB}{{\mathfrak B}}
\newcommand{\frkc}{{\mathfrak c}}    \newcommand{\frkC}{{\mathfrak C}}
\newcommand{\frkd}{{\mathfrak d}}    \newcommand{\frkD}{{\mathfrak D}}
\newcommand{\frke}{{\mathfrak e}}    \newcommand{\frkE}{{\mathfrak E}}
\newcommand{\frkf}{{\mathfrak f}}    \newcommand{\frkF}{{\mathfrak F}}
\newcommand{\frkg}{{\mathfrak g}}    \newcommand{\frkG}{{\mathfrak G}}
\newcommand{\frkh}{{\mathfrak h}}    \newcommand{\frkH}{{\mathfrak H}}
\newcommand{\frki}{{\mathfrak i}}    \newcommand{\frkI}{{\mathfrak I}}
\newcommand{\frkj}{{\mathfrak j}}    \newcommand{\frkJ}{{\mathfrak J}}
\newcommand{\frkk}{{\mathfrak k}}    \newcommand{\frkK}{{\mathfrak K}}
\newcommand{\frkl}{{\mathfrak l}}    \newcommand{\frkL}{{\mathfrak L}}
\newcommand{\frkm}{{\mathfrak m}}    \newcommand{\frkM}{{\mathfrak M}}
\newcommand{\frkn}{{\mathfrak n}}    \newcommand{\frkN}{{\mathfrak N}}
\newcommand{\frko}{{\mathfrak o}}    \newcommand{\frkO}{{\mathfrak O}}
\newcommand{\frkp}{{\mathfrak p}}    \newcommand{\frkP}{{\mathfrak P}}
\newcommand{\frkq}{{\mathfrak q}}    \newcommand{\frkQ}{{\mathfrak Q}}
\newcommand{\frkr}{{\mathfrak r}}    \newcommand{\frkR}{{\mathfrak R}}
\newcommand{\frks}{{\mathfrak s}}    \newcommand{\frkS}{{\mathfrak S}}
\newcommand{\frkt}{{\mathfrak t}}    \newcommand{\frkT}{{\mathfrak T}}
\newcommand{\frku}{{\mathfrak u}}    \newcommand{\frkU}{{\mathfrak U}}
\newcommand{\frkv}{{\mathfrak v}}    \newcommand{\frkV}{{\mathfrak V}}
\newcommand{\frkw}{{\mathfrak w}}    \newcommand{\frkW}{{\mathfrak W}}
\newcommand{\frkx}{{\mathfrak x}}    \newcommand{\frkX}{{\mathfrak X}}
\newcommand{\frky}{{\mathfrak y}}    \newcommand{\frkY}{{\mathfrak Y}}
\newcommand{\frkz}{{\mathfrak z}}    \newcommand{\frkZ}{{\mathfrak Z}}

%math boldface latters

\newcommand{\bfa}{{\mathbf a}}    \newcommand{\bfA}{{\mathbf A}}
\newcommand{\bfb}{{\mathbf b}}    \newcommand{\bfB}{{\mathbf B}}
\newcommand{\bfc}{{\mathbf c}}    \newcommand{\bfC}{{\mathbf C}}
\newcommand{\bfd}{{\mathbf d}}    \newcommand{\bfD}{{\mathbf D}}
\newcommand{\bfe}{{\mathbf e}}    \newcommand{\bfE}{{\mathbf E}}
\newcommand{\bff}{{\mathbf f}}    \newcommand{\bfF}{{\mathbf F}}
\newcommand{\bfg}{{\mathbf g}}    \newcommand{\bfG}{{\mathbf G}}
\newcommand{\bfh}{{\mathbf h}}    \newcommand{\bfH}{{\mathbf H}}
\newcommand{\bfi}{{\mathbf i}}    \newcommand{\bfI}{{\mathbf I}}
\newcommand{\bfj}{{\mathbf j}}    \newcommand{\bfJ}{{\mathbf J}}
\newcommand{\bfk}{{\mathbf k}}    \newcommand{\bfK}{{\mathbf K}}
\newcommand{\bfl}{{\mathbf l}}    \newcommand{\bfL}{{\mathbf L}}
\newcommand{\bfm}{{\mathbf m}}    \newcommand{\bfM}{{\mathbf M}}
\newcommand{\bfn}{{\mathbf n}}    \newcommand{\bfN}{{\mathbf N}}
\newcommand{\bfo}{{\mathbf o}}    \newcommand{\bfO}{{\mathbf O}}
\newcommand{\bfp}{{\mathbf p}}    \newcommand{\bfP}{{\mathbf P}}
\newcommand{\bfq}{{\mathbf q}}    \newcommand{\bfQ}{{\mathbf Q}}
\newcommand{\bfr}{{\mathbf r}}    \newcommand{\bfR}{{\mathbf R}}
\newcommand{\bfs}{{\mathbf s}}    \newcommand{\bfS}{{\mathbf S}}
\newcommand{\bft}{{\mathbf t}}    \newcommand{\bfT}{{\mathbf T}}
\newcommand{\bfu}{{\mathbf u}}    \newcommand{\bfU}{{\mathbf U}}
\newcommand{\bfv}{{\mathbf v}}    \newcommand{\bfV}{{\mathbf V}}
\newcommand{\bfw}{{\mathbf w}}    \newcommand{\bfW}{{\mathbf W}}
\newcommand{\bfx}{{\mathbf x}}    \newcommand{\bfX}{{\mathbf X}}
\newcommand{\bfy}{{\mathbf y}}    \newcommand{\bfY}{{\mathbf Y}}
\newcommand{\bfz}{{\mathbf z}}    \newcommand{\bfZ}{{\mathbf Z}}

%caligraphic letters

\newcommand{\cala}{{\mathcal A}}
\newcommand{\calb}{{\mathcal B}}
\newcommand{\calc}{{\mathcal C}}
\newcommand{\cald}{{\mathcal D}}
\newcommand{\cale}{{\mathcal E}}
\newcommand{\calf}{{\mathcal F}}
\newcommand{\calg}{{\mathcal G}}
\newcommand{\calh}{{\mathcal H}}
\newcommand{\cali}{{\mathcal I}}
\newcommand{\calj}{{\mathcal J}}
\newcommand{\calk}{{\mathcal K}}
\newcommand{\call}{{\mathcal L}}
\newcommand{\calm}{{\mathcal M}}
\newcommand{\caln}{{\mathcal N}}
\newcommand{\calo}{{\mathcal O}}
\newcommand{\calp}{{\mathcal P}}
\newcommand{\calq}{{\mathcal Q}}
\newcommand{\calr}{{\mathcal R}}
\newcommand{\cals}{{\mathcal S}}
\newcommand{\calt}{{\mathcal T}}
\newcommand{\calu}{{\mathcal U}}
\newcommand{\calv}{{\mathcal V}}
\newcommand{\calw}{{\mathcal W}}
\newcommand{\calx}{{\mathcal X}}
\newcommand{\caly}{{\mathcal Y}}
\newcommand{\calz}{{\mathcal Z}}

%math script

\newcommand{\scra}{{\mathscr A}}
\newcommand{\scrb}{{\mathscr B}}
\newcommand{\scrc}{{\mathscr C}}
\newcommand{\scrd}{{\mathscr D}}
\newcommand{\scre}{{\mathscr E}}
\newcommand{\scrf}{{\mathscr F}}
\newcommand{\scrg}{{\mathscr G}}
\newcommand{\scrh}{{\mathscr H}}
\newcommand{\scri}{{\mathscr I}}
\newcommand{\scrj}{{\mathscr J}}
\newcommand{\scrk}{{\mathscr K}}
\newcommand{\scrl}{{\mathscr L}}
\newcommand{\scrm}{{\mathscr M}}
\newcommand{\scrn}{{\mathscr N}}
\newcommand{\scro}{{\mathscr O}}
\newcommand{\scrp}{{\mathscr P}}
\newcommand{\scrq}{{\mathscr Q}}
\newcommand{\scrr}{{\mathscr R}}
\newcommand{\scrs}{{\mathscr S}}
\newcommand{\scrt}{{\mathscr T}}
\newcommand{\scru}{{\mathscr U}}
\newcommand{\scrv}{{\mathscr V}}
\newcommand{\scrw}{{\mathscr W}}
\newcommand{\scrx}{{\mathscr X}}
\newcommand{\scry}{{\mathscr Y}}
\newcommand{\scrz}{{\mathscr Z}}

%math Bbb

\newcommand{\AAA}{{\mathbb A}} %not \AA
\newcommand{\BB}{{\mathbb B}}
\newcommand{\CC}{{\mathbb C}}
\newcommand{\DD}{{\mathbb D}}
\newcommand{\EE}{{\mathbb E}}
\newcommand{\FF}{{\mathbb F}}
\newcommand{\GG}{{\mathbb G}}
\newcommand{\HH}{{\mathbb H}}
\newcommand{\II}{{\mathbb I}}
\newcommand{\JJ}{{\mathbb J}}
\newcommand{\KK}{{\mathbb K}}
\newcommand{\LL}{{\mathbb L}}
\newcommand{\MM}{{\mathbb M}}
\newcommand{\NN}{{\mathbb N}}
\newcommand{\OO}{{\mathbb O}}
\newcommand{\PP}{{\mathbb P}}
\newcommand{\QQ}{{\mathbb Q}}
\newcommand{\RR}{{\mathbb R}}
\newcommand{\SSS}{{\mathbb S}} %not \SS
\newcommand{\TT}{{\mathbb T}}
\newcommand{\UU}{{\mathbb U}}
\newcommand{\VV}{{\mathbb V}}
\newcommand{\WW}{{\mathbb W}}
\newcommand{\XX}{{\mathbb X}}
\newcommand{\YY}{{\mathbb Y}}
\newcommand{\ZZ}{{\mathbb Z}}

\newcommand{\phm}{\phantom}
\newcommand{\ds}{\displaystyle }
\newcommand{\smallstrut}{\vphantom{\vrule height 3pt }}
\def\bdm #1#2#3#4{\left(
\begin{array} {c|c}{\ds{#1}}
 & {\ds{#2}} \\ \hline
{\ds{#3}\vphantom{\ds{#3}^1}} &  {\ds{#4}}
\end{array}
\right)}
\newcommand{\wtd}{\widetilde }
\newcommand{\bsl}{\backslash }
\newcommand{\GL}{{\mathrm{GL}}}
\newcommand{\SL}{{\mathrm{SL}}}
\newcommand{\GSp}{{\mathrm{GSp}}}
\newcommand{\PGSp}{{\mathrm{PGSp}}}
\newcommand{\SP}{{\mathrm{Sp}}}
\newcommand{\SO}{{\mathrm{SO}}}
\newcommand{\SU}{{\mathrm{SU}}}
\newcommand{\Ind}{\mathrm{Ind}}
\newcommand{\Hom}{{\mathrm{Hom}}}
\newcommand{\Ad}{{\mathrm{Ad}}}
\newcommand{\Sym}{{\mathrm{Sym}}}
\newcommand{\Mat}{\mathrm{M}}
\newcommand{\sgn}{\mathrm{sgn}}
\newcommand{\trs}{\,^t\!}
\newcommand{\iu}{\sqrt{-1}}
\newcommand{\oo}{\hbox{\bf 0}}
\newcommand{\ono}{\hbox{\bf 1}}
\newcommand{\smallcirc}{\lower .3em \hbox{\rm\char'27}\!}
\newcommand{\bAf}{\bA_{\hbox{\eightrm f}}}
\newcommand{\thalf}{{\textstyle{\frac12}}}
\newcommand{\shp}{\hbox{\rm\char'43}}
\newcommand{\Gal}{\operatorname{Gal}}

\newcommand{\bdel}{{\boldsymbol{\delta}}}
\newcommand{\bchi}{{\boldsymbol{\chi}}}
\newcommand{\bgam}{{\boldsymbol{\gamma}}}
\newcommand{\bome}{{\boldsymbol{\omega}}}
\newcommand{\bpsi}{{\boldsymbol{\psi}}}
\newcommand{\GK}{\mathrm{GK}}
\newcommand{\ord}{\mathrm{ord}}
\newcommand{\diag}{\mathrm{diag}}
\newcommand{\ua}{{\underline{a}}}
\newcommand{\ZZn}{\ZZ_{\geq 0}^n}
\newcommand{\calhnd}{{\mathcal H}^\mathrm{nd}}
\newcommand{\EGK}{\mathrm{EGK}}

%%%added section chlee
\newcommand{\Zp}{\ZZ_p}
\newcommand{\Qp}{\QQ_p}
\def\mat#1#2{\calhnd_{#1}(\ZZ_{#2})}
\def\Zmat#1{\calhnd_{#1}(\ZZ)}
\def\subsuperscript#1#2#3{#1_{#2}^{#3}}
\def\ordet#1{\ord(\det #1)}
\newcommand{\intmult}{(T_{m_1} \cdot T_{m_2}\cdot T_{m_3})_{S}}
\newcommand{\vecm}{\underline{m}}
\def\LB#1{L_2(\vecm;#1)}
\newcommand{\LA}{L_1(\vecm)}
\newcommand{\LBp}{\LB{p}}
\newcommand{\LBB}{L_2(\vecm)}
%%%

\theoremstyle{plain}
\newtheorem{theorem}{Theorem}[section]
\newtheorem{lemma}{Lemma}[section]
\newtheorem{proposition}{Proposition}[section]
\newtheorem{problem}{Problem}[section]
\theoremstyle{definition}
\newtheorem{definition}{Definition}[section]
\newtheorem{conjecture}{Conjecture}[section]
\newtheorem{remark}{\textbf{Remark}}[section]
\newtheorem{corollary}{\textbf{Corollary}}[section]

\newtheorem{mainthm}{\textbf{Theorem}}
\renewcommand{\themainthm}{}
\newtheorem{example}{Example}[section]

%%%%%%%%%%%%%%%%title%%%%%%%%%%%%%%%%%%%
\title[]{An explicit formula for the extended Gross-Keating datum of a quadratic form}
%{Remarks on  the Extended Gross-Keating data and the Siegel series of a quadratic form}
\author[S. Cho]{Sungmun Cho}
\address{Department of Mathematics, POSTECH, 77, Cheongam-ro, Nam-gu, Pohang-si, Gyeongsangbuk-do, 37673, KOREA}
\email{sungmuncho12@gmail.com}
\author[T. Ikeda]{Tamotsu Ikeda}
\address{Graduate school of mathematics, Kyoto University, Kitashirakawa, Kyoto, 606-8502, Japan}
\email{ikeda@math.kyoto-u.ac.jp}
\author[H. Katsurada]{Hidenori Katsurada}
\address{Muroran Institute of Technology,
27-1 Mizumoto, Muroran, 050-8585, Japan}
\email{hidenori@mmm.muroran-it.ac.jp}
\author[C.-h. Lee]{Chul-hee Lee}
\address{School of Mathematics, Korea Institute for Advanced Study, Seoul 02455, Korea}
\email{chlee@kias.re.kr}
\author[T. Yamauchi]{Takuya Yamauchi}
\address{Mathematical Institute of Tohoku University, 
6-3 Aoba, Aramaki, Aoba-Ku, Sendai ,980-8578, Japan}
\email{yamauchi@math.tohoku.ac.jp}

\begin{abstract}
%The Gross-Keating invariant is a key ingredient to compare the arithmetic intersection number on orthogonal Shimura varieties with the central derivative of the Fourier coefficients of the Siegel-Eisenstein series in the context of Kudla's program.
In this paper, we give a formula for the extended Gross-Keating datum of a quadratic form defined over a finite extension of $\ZZ_p$ (for $p>2$) or a finite unramified extension of $\ZZ_2$. 
As an application, we describe an explicit formula for the Siegel series for $\ZZ_p$.
We also present the details of algorithms implemented in a Mathematica package to compute the extended Gross-Keating datum and the Siegel series.
%We also have built the Mathematica package \textit{computeGK} to compute the extended Gross-Keating datum and the Siegel series.
%Gross-Keating invariants.
%\textit{computeGK}.
\end{abstract}
\maketitle

%%%%%%%%%%%%%%%%%%%%%%%%%%%%%%

\tableofcontents

\section*{Introduction}
{\bf Background.}
In 1993, B. Gross and K. Keating (cf. \cite{GK}) computed certain arithmetic intersection number in the self-product $Y_0(1)\times Y_0(1)/ \ZZ$ of the moduli stack $Y_0(1)$ of elliptic curves over $\ZZ$. 
One amazing fact in their work is to describe it purely in terms of (Euler product and) certain invariant of a ternary quadratic form over $\ZZ_p$ invented by themselves.
This invariant is later generalized to a quadratic lattice (or a half-integral symmetric matrix) of any degree defined over a finite extension of $\ZZ_p$.  It is nowadays called the Gross-Keating invariant.
The Gross-Keating invariant had been almost forgotten\footnote{The Gross-Keating invariant had been treated in \cite{Bouw}. The first sentence of loc. cit. says `This note provides details on \cite{GK} Section 4.' We also note that Wedhorn \cite{Wedhorn} reformulated the explicit  formula of the Siegel series in \cite{Kat} for  the ternary case in   terms of Gross-Keating invariant.  } for a while after the work of Gross and Keating.
Recently, in \cite{IK2}, two (Ikeda and Katsurada) of us proved that the local factor of the Fourier coefficient of the Siegel-Eisenstein series, which is called  the Siegel series,
% associated to a local quadratic lattice of any degree 
is completely determined by the Gross-Keating invariant imposed with additional datum, called \textit{the extended GK datum}.
Using this, another two (Cho and Yamauchi) of us reformulated the Siegel series in terms of smoothening of certain integral schemes in \cite{CY}.
This work has been  used as a key ingredient\footnote{This terminology is taken from page 7 of \cite{lz}.} in a recent proof of Kudla-Rapoport conjecture by C. Li and W. Zhang in \cite{lz}, which compares the arithmetic intersection multiplicity of special cycles on (unitary or orthogonal) Rapoport-Zink spaces with the derivative of the Siegel series in the context of Kudla's program.

%We provide more background on the extended GK datum. 
Let us explain in which context the notion of the extended GK datum first appeared. 
Let $F$ be a non-archimedean local field of characteristic $0$, and $\frko_F$ the ring of integers in $F$.
In \cite{Kat}, the third named author gave an explicit formula for the Siegel series in the case $\frko_F=\ZZ_p$.
%of a half-integral symmetric matrix over $\ZZ_p$.
However, the formula is complicated at a glance in the case $p=2$ and it is not clear which invariant of a half-integral symmetric matrix determines the Siegel series. 
To overcome these, in \cite{IK1} we introduced the extended GK datum, which is an extended version of the Gross-Keating invariant, and expressed the Siegel series of a half-integral symmetric matrix $B$ over  $\frko_F$ for any $F$ in terms of it in \cite{IK2}. 

In \cite{IK1}, the extended GK datum  of $B$ is shown to be obtained through a Jordan decomposition of $B$ in the non-dyadic case,   and also through an optimal form (see Definition \ref{def.1.2}) or a reduced form (see Definition \ref{def.1.8}) of $B$ in the dyadic case. However, there is currently no effective algorithm for obtaining  the latter two, which makes it difficult to compute the extended GK datum using them in the dyadic case.
\\

{\bf Main results.}
% As explained in the above subsection, 
% it is an important problem to give an explicit formula for the Siegel series of  a half-integral symmetric matrix (or a quadratic lattice) over $\frko_F$.
Our main result is 
to give a formula\footnote{This formula is used in \cite{Cho1} and \cite{CY}.} for the extended GK datum, denoted by $\EGK(B)$, and for a naive EGK datum
of a half-integral symmetric matrix $B$ in the case that $F$ is a finite unramified extension of $\QQ_2$ (cf. Theorems \ref{th.4.2}-\ref{th.4.3}), in terms of the newly introduced notion of pre-optimal forms (cf. Definition \ref{def.2.2}). 
We note that
these two data can easily be given in the case $F$ is a non-dyadic field
(cf. Theorem \ref{th.4.4}).  
Remarkably,  $\EGK(B)$ is given through a pre-optimal form of $B$, which is more simply obtained than an optimal form or a reduced form of $B$. 
%We refer to Definitions \ref{def.1.1}-\ref{def.1.2} for the notion of optimal form.

%As an application,  in the case $F=\QQ_p$, we  describe an explicit formula for the Siegel series $b(B,s)$ of $B$  in terms of a naive EGK datum of $B$ (cf. Theorem \ref{th.5.5}). This formula essentially coincides with \cite[Theorem 4.3]{Kat}, but it consequently gives a much simpler and strikingly uniform, independent of the parity of $p$, inductive method of computing the Siegel series.
%%the former is canonical whereas the latter depends on a chosen matrix. 
%We note that such a formula is generalized to any non-archimedean local field of characteristic $0$ in more sophisticated manner in \cite[Theorem 1.1]{IK2}. However, our formula has the advantage of being effective for computing the Siegel series because we have an explicit formula for obtaining naive EGK data based on Theorems \ref{th.3.1} and \ref{th.4.3}.

As an application,  in the case $F=\QQ_p$, we describe an explicit formula for the Siegel series $b(B,s)$ of $B$  in terms of a naive EGK datum of $B$ (cf. Theorem \ref{th.5.5}). This formula essentially coincides with \cite[Theorem 4.3]{Kat} which has been generalized to an arbitrary non-archimedean local field of characteristic $0$ in more sophisticated manner in \cite[Theorem 1.1]{IK2}. Due to Theorems \ref{th.3.1} and \ref{th.4.3}, however, our formula has the advantage of being effective for computing the Siegel series, and consequently takes a simpler and strikingly unified form (cf. $\calf(H;X)$ in Definition \ref{def.5.2}), independent of the parity of $p$.

%These formulas 
Finally, based on the above results, we have built the Mathematica package \textit{computeGK}\footnote{The code is available at \url{https://github.com/chlee-0/computeGK}.}. This is the first computer program devoted to the Gross-Keating invariant. The program enables to compute a pre-optimal form, a naive extended Gross-Keating datum, and the Siegel series of a half-integral symmetric matrix over $\ZZ_p$. Indeed, by using this program, the third named  author could compute the Fourier coefficients 
of the Klingen-Eisenstein lift of a cuspidal Hecke eigenform for $Sp_2(\ZZ)$ to a  (vector valued) modular form for $Sp_4(\ZZ)$ for the first time, which plays a crucial role in proving the congruence  conjectured by Harder \cite{Ha} in some cases (cf. \cite{Kat2}). Note that there is a LISP code for the Siegel series, used in \cite{MR1954971, MR3739221} based on the formulas of \cite{Kat}, but it is more specialized in computing the Fourier coefficients of the Siegel-Eisenstein series.
%Another important application of  the formula for the Gross-Keating invariant is to conceptually study  the relation between    (arithmetic and mod $p$) intersection numbers on orthogonal Shimura varieties and the Fourier coefficients (of the derivative) of the Siegel Eisenstein series in the context of Kudla's program, because both sides would be  formulated in terms of the Gross-Keating invariants. For more discussions in this direction, we refer to   \cite{CY} written by two of us.
\\

{\bf The structure of the paper.}
The paper is organized as follows.
In Section 1,  we recall the Gross-Keating invariants and the extended GK datum following \cite{IK1}. In Section 2, we introduce the notion of pre-optimal forms, and give a key theorem for obtaining an induction formula for  naive EGK data of 
pre-optimal forms (cf. Theorem \ref{th.2.1}) in the case
$F$ is a finite unramified extension of $\QQ_2$.
In Section 3, we give an explicit formula for the Gross-Keating invariant.
In Section 4, we recall EGK data and naive EGK data, and give an explicit  formula for naive EGK datum of a half-integral symmetric matrix in the case $F$ is a finite unramified extension of $\QQ_2$. 
In Section 5, we give an explicit formula for the Siegel series in terms a naive EGK data in the case $F=\QQ_p$.
In Section 6, we discuss the detailed procedures to obtain a weak canonical decomposition for the purpose of algorithmic computation in the case $F=\QQ_2$.
%for arbitrarily given $B\in \mat{n}{2}$ 
In Appendix A, we provide the tables of arithmetic intersection numbers of three modular correspondences, generated by \textit{computeGK}.

\bigskip

{\bf Acknowledgments.}
This research was partially supported by the JSPS KAKENHI Grant Number 16F16316, 25247001, 17H02834, and 16H03919,  KIAS Grant MG067301, Samsung Science and Technology
Foundation under Project Number SSTF-BA1802-03, and NRF-2018R1A4A1023590.
We thank Cris Poor, Jerry Shurman and David Yuen for allowing their code for Jordan decomposition to be included in the \textit{computeGK } program with modifications. 
Lee would like to thank Richard Borcherds for introducing the results of \cite{Kat} to him.
\bigskip

{\bf Notation.} Let $F$ be a non-archimedean local field of characteristic $0$, and $\frko=\frko_F$ its ring of integers.
The maximal ideal and  the residue field of $\frko$ is denoted by $\frkp$ and $\frkk$, respectively.
We put $q=[\frko:\frkp]$.
$F$ is said to be dyadic if $q$ is even.
We fix a prime element $\vpi$ of $\frko$ once and for all.
The order of $x\in F^\times$ is given by $\mathrm{ord}(x)=n$ for $x\in \vpi^n \frko^\times$.
We understand $\ord(0)=+\infty$.
Put $F^{\times 2}=\{x^2\,|\, x\in F^\times\}$.
Similarly, we put $\frko^{\times 2}=\{x^2\,|\, x\in\frko^\times\}$.
When $R$ is a ring, the set of $m\times n$ matrices with entry in $R$ is denoted by $\mathrm{M}_{mn}(R)$ or $\mathrm{M}_{m,n}(R)$.
As usual, $\mathrm{M}_n(R)=\mathrm{M}_{n,n}(R)$.
The identity matrix of degree  $n$ is denoted by $\mathbf{1}_n$.
For $X_1\in \mathrm{M}_s(R)$ and $X_2\in\mathrm{M}_t(R)$, the matrix $\begin{pmatrix} X_1 & 0 \\ 0 & X_2\end{pmatrix}\in\mathrm{M}_{s+t}(R)$ is denoted by $X_1\perp X_2$.
The diagonal matrix whose diagonal entries are $b_1$, $\ldots$, $b_n$ is denoted by $\mathrm{diag}(b_1, \dots, b_n)=(b_1)\perp\dots\perp (b_n)$.

For two sequences $\AAA=(A_1,\ldots,A_r)$ and $\BB=(B_1,\ldots,B_s)$ of numbers or sets, we denote by $(\AAA,\BB)$ the sequence $(A_1,\ldots,A_r,B_1,\ldots,B_s)$.
\section{Extended GK datum}

The set of symmetric matrices $B\in \mathrm{M}_n(F)$ of degree $n$ is denoted by $\mathrm{Sym}_n(F)$.
For $B\in \mathrm{Sym}_n(F)$ and $X\in\GL_n(F)$, we set $B[X]={}^t\! XBX$.
When $G$ is a subgroup of $\GL_n(F)$, we shall say that two elements $B_1, B_2\in\mathrm{Sym}_n(F)$ are called $G$-equivalent, if there is an element $X\in G$ such that $B_1[X]=B_2$.
We say that $B=(b_{ij})\in \mathrm{Sym}_n(F)$ is a half-integral symmetric matrix if
\begin{align*}
b_{ii}\in\frko_F &\qquad (1\leq i\leq n),  \\
2b_{ij}\in\frko_F& \qquad (1\leq i\leq j\leq n).
\end{align*}
The set of all half-integral symmetric matrices of degree $n$ is denoted by $\calh_n(\frko)$.
An element $B\in\calh_n(\frko)$ is non-degenerate if $\det B\neq 0$.
The set of all non-degenerate elements of $\calh_n(\frko)$ is denoted by $\calh_n(\frko)^{\rm nd}$.
For $B=(b_{ij})_{1\leq i,j\leq n}\in\calh_n(\frko)$ and $1\leq m\leq n$,  we denote the upper left $m\times m$ submatrix $(b_{ij})_{1\leq i, j\leq m}\in\calh_m(\frko)$ by $B^{(m)}$.

When two elements $B, B'\in\calh_n(\frko)$ are $\GL_n(\frko)$-equivalent, we just say they are equivalent and write $B\sim B'$.
The equivalence class of $B$ is denoted by $\{B\}$, i.e., 
$\{B\}=\{B[U]\,|\, U\in\GL_n(\frko)\}$.

\begin{definition}
\label{def.1.1}
Let $B=(b_{ij})\in\calh_n(\frko)^{\rm nd}$.
Let $S(B)$ be the set of all non-decreasing sequences $(a_1, \ldots, a_n)\in\ZZn$ such that
\begin{align*}
&\ord(b_{ii})\geq a_i \qquad\qquad\qquad\quad (1\leq i\leq n), \\
&\ord(2 b_{ij})\geq (a_i+a_j)/2  \qquad\; (1\leq i\leq j\leq n).
\end{align*}
Put
\[
\bfS(\{B\})=\bigcup_{B'\in\{B\}} S(B')=\bigcup_{U\in\GL_n(\frko)} S(B[U]).
\]
The Gross-Keating invariant $\GK(B)$ of $B$ is the greatest element of $\bfS(\{B\})$ with respect to the lexicographic order $\succeq$ on $\ZZn$.
\end{definition}
Here, the lexicographic order $\succeq$ is, as usual, defined as follows.
For distinct sequences $(y_1, y_2, \ldots, y_n),$  $(z_1, z_2, \ldots, z_n)\in \ZZ_{\geq 0}^n$, let $j$ be the largest integer such that $y_i=z_i$ for $i<j$.
Then $(y_1, y_2, \ldots, y_n)\succneqq  (z_1, z_2, \ldots, z_n)$ if $y_j>z_j$.
We define $(y_1, y_2, \ldots, y_n)\succeq  (z_1, z_2, \ldots, z_n)$ if $(y_1, y_2, \ldots, y_n)\succneqq  (z_1, z_2, \ldots, z_n)$ or $(y_1, y_2, \ldots, y_n) = (z_1, z_2, \ldots, z_n)$.

A sequence of length $0$ is denoted by $\emptyset$.
When $B$ is the empty matrix, we understand $\GK(B)=\emptyset$.
By definition, the Gross-Keating invariant $\GK(B)$ is determined only by the equivalence class of $B$.
Note that $\GK(B)=(a_1, \ldots, a_n)$ is also defined by
\begin{align*}
a_1&=\max_{(y_1,  \ldots)\in \bfS(\{B\})} \,\{y_1\}, \\
a_2&=\max_{(a_1, y_2, \ldots)\in \bfS(\{B\})}\, \{y_2\}, \\
&\cdots \\
a_n&=\max_{(a_1, a_2, \ldots, a_{n-1}, y_n)\in \bfS(\{B\})}\, \{y_n\}.
\end{align*}

\begin{definition} 
\label{def.1.2}
$B\in\calh_n(\frko)$ is optimal if $\GK(B)\in S(B)$.
\end{definition}
By definition, a non-degenerate half-integral symmetric matrix $B\in\calh_n(\frko)^{\rm nd}$ is equivalent to an optimal form.

Let $L$ be a free module of rank $n$ over $\frko$, and $Q$ a $\frko$-valued quadratic form on $L$.
The pair $(L, Q)$ is called a quadratic module over $\frko$.
The symmetric bilinear form $(x, y)_Q$ associated to $Q$ is defined by
\[
 (x, y)_Q=Q(x+y)-Q(x)-Q(y), \qquad x, y\in L.
\]
When there is no fear of confusion, $(x, y)_Q$ is simply denoted by $(x, y)$.
If $\underline{\psi}=\{\psi_1, \ldots, \psi_n\}$ is an ordered basis of $L$, we call the triple $(L, Q, \underline{\psi})$ a framed quadratic $\frko$-module.
Hereafter, ``a basis'' means an ordered basis.
For a framed quadratic $\frko$-module $(L, Q, \underline{\psi})$, we define a matrix $B=(b_{ij})\in\calh_n(\frko)$ by
\[
 b_{ij}=\frac12 (\psi_i, \psi_j).
\]
The isomorphism class of $(L, Q, \underline{\psi})$ (as a framed quadratic $\frko$-module) is determined by $B$.
We say that $B\in\calh_n(\frko)$ is associated to the framed quadratic module $(L, Q, \underline{\psi})$.
If $B$ is non-degenerate, we also say $(L, Q)$ or $(L, Q, \underline{\psi})$ is non-degenerate.
The set $S(B)$ is also denoted by $S(\underline{\psi})$.
If $B$ is optimal, then $\underline{\psi}$ is called an optimal basis.
%The Gross-Keating invariant $\GK(B)$ is also denoted by $\GK(\underline{\psi})$.
We consider $\mathrm{Aut}(L)$ acting on $L$ from the right.
We note that  the equivalence class of $B$ is determined by the isomorphism class of the quadratic modules $(L, Q)$. Conversely, for $B \in \calh_n(\frko)$, we can take a framed quadratic $\frko$-module $(L,Q,\underline{\psi})$  such  that  $B$ is associated to $(L,Q,\underline{\psi})$. 
Then we write $L$ as $L_B$. Let $B=B_1 \bot B_2$ with $B_i \in \calh_n(\frko) \ (i=1,2)$. Then $L_{B_i}$ can be regarded as a $\frko$-submodule of $L_{B}$ in  a natural way.

For $B\in\calh_n(\frko)^{\rm nd}$, we put $D_B=(-4)^{[n/2]}\det B$.
If $n$ is even, we denote the discriminant ideal of $F(\sqrt{D_B})/F$ by $\frkD_B$. Put 
\[\frke_B=
\begin{cases}
\ord(D_B)-\ord(\frkD_B)   & \text{ if $n$ is even} \\
\ord(D_B)                         & \text{ if $n$ is odd.}
\end{cases}\]
We also put
\[
\xi_B=
\begin{cases} 
1 & \text{ if $D_B\in F^{\times 2}$,} \\
-1 & \text{ if $F(\sqrt{D_B})/F$ is unramified and $[F(\sqrt{D_B}):F]=2$,} \\
0 & \text{ if $F(\sqrt{D_B})/F$ is ramified.} 
\end{cases}
\]
\begin{definition}\label{def.1.3}
For $B\in\calh_n(\frko)^{\rm nd}$, we put
\[
\Del(B)=
\begin{cases}
\ord(D_B) & \text{ if $n$ is odd,} \\
\ord(D_B)-\ord(\mathfrak{D}_B)+1-\xi_B^2 & \text{ if $n$ is even.}
\end{cases}
\]
\end{definition}
Note that if $n$ is even, then
\[
\Del(B)=
\begin{cases}
\ord(D_B) & \text{ if $\ord(\mathfrak{D}_B)=0$,} \\
\ord(D_B)-\ord(\mathfrak{D}_B)+1 & \text{ if $\ord(\mathfrak{D}_B)>0$.}
\end{cases}
\]

For $\ua=(a_1, a_2, \ldots, a_n)\in\ZZn$, we write $|\ua|=a_1+a_2+\cdots+a_n$. We review some results in \cite{IK1}.

\begin{theorem} 
\label{th.1.1} {\rm(\cite{IK1}, Theorem 0.1)}
For $B\in\calh_n(\frko)^{\rm nd}$,  we have
\[
|\GK(B)|=\Del(B).
\]
\end{theorem}

For a non-decreasing sequence $\ua=(a_1, a_2, \ldots, a_n)\in\ZZn$, we set
\[
G_{\ua}=
\{g=(g_{ij})\in \GL_n(\frko) \,|\, \text{ $\ord(g_{ij})\geq (a_j-a_i)/2$ \ if $a_i<a_j$}\}.
\]

\begin{theorem} 
\label{th.1.2} {\rm \cite{IK1}, Theorem 0.2)}
Suppose that $B\in\calhnd_n(\frko)$ is optimal and  $\GK(B)=\ua$.
Let $U\in \GL_n(\frko)$.
Then $B[U]$ is optimal if and only if $U\in G_\ua$.
\end{theorem}

For $\underline{a}=(a_1, a_2, \ldots, a_n)\in\ZZ_{\geq 0}^n$, we put $\underline{a}^{(m)}=(a_1, a_2, \ldots, a_m)$ for $m\leq n$.

\begin{theorem} 
\label{th.1.3}{\rm \cite{IK1}, Theorem 0.3)}
Suppose that $B\in\calh_n(\frko)$ is optimal and  $\GK(B)=\ua$.
If $a_k<a_{k+1}$, then $B^{(k)}$ is also optimal and $\GK(B^{(k)})=\ua^{(k)}$.
\end{theorem}

\begin{definition} 
\label{def.1.4}
The Clifford invariant (see Scharlau \cite{scharlau}, p.~333) of $B\in\calhnd_n(\frko)$ is the Hasse invariant of the Clifford algebra (resp.~the even Clifford algebra) of $B$ if $n$ is even (resp.~odd).
\end{definition}
We denote the Clifford invariant of $B$ by $\eta_B$.
If $B$ is $\GL_n(F)$-equivalent to $\mathrm{diag}(b'_1, \ldots, b'_n)$, then 
\begin{align*}
\eta_B
=&
\langle -1, -1 \rangle^{[(n+1)/4]}\langle -1, \det B \rangle^{[(n-1)/2]} \prod_{i < j} \langle b'_i, b'_j \rangle \\
=&\begin{cases}
\langle -1, -1 \rangle^{m(m-1)/2}\langle -1, \det B \rangle^{m-1} \ds\prod_{i < j} \langle b'_i, b'_j \rangle & \text{ if $n=2m$, } \\
\noalign{\vskip 6pt}
\langle -1, -1 \rangle^{m(m+1)/2}\langle -1, \det B \rangle^{m} \ds\prod_{i < j} \langle b'_i, b'_j \rangle & \text{ if $n=2m+1$. } 
\end{cases}
\end{align*}
(See Scharlau \cite{scharlau} pp.~80--81.)
If $H\in\calhnd_2(\frko)$ is $\GL_2(F)$-isomorphic to a hyperbolic plane, then $\eta_{B\perp H}=\eta_B$.
In particular, if $n$ is odd, then we have
\begin{align*}
\eta_B
=&
\begin{cases} 
1 & \text{ if  $B$ is split over $F$,} \\
-1 & \text{ otherwise.} 
\end{cases}
\end{align*}

\begin{theorem} 
\label{th.1.4}{\rm \cite{IK1}, Theorem 0.4)}
Let $B, B_1\in\calhnd_n(\frko)$.
Suppose that $B \sim B_1$ and both $B$ and $B_1$ are optimal.
Let $\ua=(a_1, a_2, \ldots, a_n)=\GK(B)=\GK(B_1)$.
Suppose that $a_k<a_{k+1}$ for $1\leq k < n$.
Then the following assertions (1) and (2) hold.
\begin{itemize}
\item[(1)] If $k$ is even, then $\xi_{B^{(k)}}=\xi_{B_1^{(k)}}$.
\item[(2)] If $k$ is odd, then $\eta_{B^{(k)}}=\eta_{B_1^{(k)}}$.
\end{itemize}
\end{theorem}

\begin{definition} 
\label{def.1.5}
Let  $B \in \calh_n(\frko )$ be an optimal form such that $\GK(B)=\ua$.
Write 
\[\mathrm{GK}(B)=(\underbrace{m_1,\ldots,m_1}_{n_1},\ldots,\underbrace{m_r,\ldots,m_r}_{n_r})\]
with $m_1<\cdots<m_r$ and $n=n_1+\cdots+n_{r-1}+n_r$.
For $j=1,2,\ldots,r$ put
$$n_j^\ast=\sum_{u=1}^j n_u.$$ 
We define $\zeta_s=\zeta_s(B)$ by 
\[
\zeta_s=\zeta_s(B)=
\begin{cases}
\xi_{B^{(n_s^\ast)}} & \text{ if $n_s^\ast$ is even,} \\
\eta_{B^{(n_s^\ast)}} & \text{ if $n_s^\ast$ is odd.} 
\end{cases}
\]
Then put
$\EGK(B)=(n_1,\ldots,n_r;m_1,\ldots,m_r;\zeta_1,\ldots,\zeta_r)$.
For $B \in \calhnd_n(\frko )$, we define $\EGK(B)=\EGK(B')$, where $B'$ is an optimal form equivalent to $B$.
\end{definition}
By  Theorem  \ref{th.1.4}, this definition does not depend on the choice of $B'$.
Thus, $\EGK(B)$ depends only on the equivalence class of $B$.
We call $\EGK(B)$ the extended GK datum of $B$.

From now on, we assume that $F$ is a dyadic field.  Now we recall the notion of reduced form.
We denote by $\frkS_n$ the symmetric group of degree $n$. Recall that a permutation $\sig\in \frkS_n$ is an involution if $\sig^2=\mathrm{id}$.

\begin{definition}
\label{def.1.6}

For an involution $\sig \in \frkS_n$ and a non-decreasing sequence $\ua=(a_1,\ldots,a_n)$ of non-negative integers, we set
\begin{align*}
\calp^0&=\calp^0(\sigma)=\{i\,| 1\leq i\leq n,\;  i=\sig(i)\}, \\
\calp^+&=\calp^+(\sigma)=\{i\,| 1\leq i\leq n,\;  a_i>a_{\sig(i)}\}, \\
\calp^-&=\calp^-(\sigma)=\{i\,| 1\leq i\leq n,\;  a_i<a_{\sig(i)}\}. 
\end{align*}
Write $\ua$ as 
\[\ua=(\underbrace{m_1,\ldots,m_1}_{n_1},\ldots,\underbrace{m_r,\ldots,m_r}_{n_r})\]
with $m_1<\cdots<m_r$ and $n=n_1+\cdots+n_{r-1}+n_r$.
For $j=1,2,\ldots,r$ put
$$n_j^\ast=\sum_{u=1}^j n_u$$
and $$I_s=\{n_{s-1}^\ast+1,\ldots,n_s^\ast\}.$$
We  say that an involution $\sig\in\frkS_n$ is an $\underline{a}$-admissible involution if the following two conditions are satisfied.
\begin{itemize}
\item[(i)] 
$\calp^0$ has at most two elements.
If $\calp^0$ has two distinct elements $i$ and $j$, then $a_i\not\equiv a_j \text{ mod $2$}$. 
Moreover, if $i \in  I_s\cap \calp^0$, then $i$ is the maximal element of $I_s$, and
\[i=\max\{j \ | \ j \in \calp^0 \cup \calp^+ , a_j \equiv a_i \text{ mod } 2 \}.\]
\item[(ii)]
For $s=1, \ldots, r$,  there is at most one element in $I_s\cap\calp^-$.
If $i \in  I_s\cap\calp^-$, then $i$ is the maximal element of $I_s$ and
\[
\sig(i)=\min\{j\in \calp^+ \,| \, j>i,\, a_j\equiv a_i \text{ mod } 2\}.
\]
\item[(iii)]
For $s=1, \ldots, r$,  there is at most one element in $I_s\cap\calp^+$.
If $i \in  I_s\cap\calp^+$, then $i$ is the minimal element of $I_s$ and
\[
\sig(i)=\max\{j\in \calp^- \,| \, j<i,\, a_j\equiv a_i \text{ mod } 2\}.
\]
\item[(iv)]
If $a_i=a_{\sig(i)}$, then $|i -\sig(i)| \le 1$. 
\end{itemize}
\end{definition}
This is called a standard $\underline{a}$-admissible involution in \cite{IK1}, but in this paper we omit the word ``standard'', since we do not consider an $\underline{a}$-admissible involution which is not standard.

\begin{definition}
\label{def.1.7}
For $\ua=(a_1, \ldots, a_n)\in\ZZ_{\geq 0}^n$, put
\begin{align*}
\calm(\ua)&=\left\{B=(b_{ij})\in\calh_n(\frko)\,\vrule\, \begin{array}{ll}\ord(b_{ii})\geq a_i, \\ \ord(2 b_{ij})\geq (a_i+a_j)/2,\end{array} \; (1\leq i < j\leq n)\right\},  \\
\calm^0(\ua)&=\left\{B=(b_{ij})\in\calh_n(\frko)\,\vrule\, \begin{array}{ll}\ord(b_{ii}) > a_i, \\ \ord(2 b_{ij}) > (a_i+a_j)/2,\end{array} \; (1\leq i < j\leq n)\right\}.
\end{align*}
\end{definition}

\begin{definition}
\label{def.1.8}
Let $\sig\in\frkS_n$ be an $\ua$-admissible involution.
We say that $B=(b_{ij})\in \calm(\ua)$ is a reduced form with GK-type $(\ua, \sig)$ if the following conditions are satisfied.
\begin{itemize}
\item[(1)]  If $i\notin\calp^0$ and $i\leq j=\sig(i)$, then 
\[ \begin{cases}
\ord(2 b_{i\,\sig(i)})={\ds \frac{a_i+a_{\sig(i)}}2} & \text{ if }  i\notin\calp^0, \\ 
\noalign{\vskip 6pt}
\ord(b_{ii})=a_i & \text{ if } i\in\calp^-.
 \end{cases}
 \]
\item[(2)] If $i\in\calp^0$, then
\[
\ord(b_{ii})=a_i.
\]
\item[(3)] If $j\neq i, \sig(i)$, then
\[
\ord(2 b_{ij})>\frac{a_i+a_j}2,
\]
\end{itemize}
\end{definition}
We often say that $B$ is a reduced form with GK-type $\ua$ without mentioning $\sig$.
We formally think of a matrix of degree $0$ as a reduced form with GK-type $\emptyset$.
The following theorems are fundamental in our theory.
\begin{theorem}
\label{th.1.5}
{\rm (\cite[Corollary 5.1]{IK1})} Let $B$ be a reduced form of $\GK$ type $(\ua,\sigma)$. Then we have $\GK(B)=\ua$.
\end{theorem}
\begin{theorem}
\label{th.1.6} 
{\rm  (\cite[Theorem 4.1]{IK1})}
Assume that $\mathrm{GK}(B)=\underline{a}$ for $B\in \calh_n(\frko)^\mathrm{nd}$.
Then $B$ is $GL_n(\frko)$-equivalent to a reduced form of $\GK$ type $(\underline{a}, \sigma)$ for some   $\underline{a}$-admissible involution $\sigma$.
\end{theorem}
By Theorem \ref{th.1.6}, any non-degenerate  half-integral symmetric matrix $B$ over $\frko$ is $GL_n(\frko)$-equivalent to  a reduced form $B'$. Then we say that $B$ has a reduced decomposition $B'$. For later purpose, we give the following:
\begin{proposition}
\label{prop.1.1}
Let $B =(b_{ij})_{n \times n}$ be a reduced form of type $(\underline{a}, \sigma)$.
\begin{itemize}
\item[(1)] Assume that $n$ is odd and let $i_0 \in \calp^0$. Then  
$$\mathrm{ord}(b_{i_0,i_0}) \equiv \mathrm{ord}(\det B) \text{ mod } 2.$$
\item[(2)] Assume that $n$ is even and $\xi_B=0$.  Then, for any integer $k$,  there is an integer $i_0 \in \calp^0$ such that   
$$\mathrm{ord}(b_{i_0,i_0}) \equiv k \text{ mod } 2.$$
\end{itemize}
\end{proposition}
\begin{proof} (1) Let $n$ be odd. By Theorem \ref{th.1.1} and the Definition \ref{def.1.8} (2), we have
\[\mathrm{ord}(\det B)=\sum_{i=1}^n a_i -(n-1) \text{ and } \mathrm{ord}(b_{i_0,i_0})=a_{i_0}.\]
Since $n$ is odd, we have $\calp^0=\{i_0 \}$, and 
$\sum_{1 \le i \le n, i \not=i_0} a_i$ is even.  This proves the assertion (1).

(2) Let $n$ be even. By Definitions \ref{def.1.6} and \ref{def.1.8}, we have $\calp^0=\{i_1,i_2 \}$, and $\ord(b_{i_1,i_1})=a_{i_1} \not\equiv a_{i_2}=\ord(b_{i_2,i_2}) \text{ mod } 2$. Thus, the assertion (2) holds.
\end{proof}

\section{Pre-optimal forms}
Throughout this section and the next, we assume that  $F$ is a finite unramified extension of $\ZZ_2$, and $\frko$ be its ring of integers in $F$. 
 We denote by $S_m(\frko)_e$ (resp. $S_m(\frko)_d$) the set of even integral symmetric (resp. diagonal) matrix of degree $m$ with entries in $\frko.$ For a sequence ${\bf a}=(a_1,\cdots,a_n) \in \ZZ^n,$ put $|{\bf a}|= a_1+\cdots+a_n,$
and ${\bf a}_i=a_i.$  For two sequences ${\bf a}=(a_1,\cdots,a_r)$ and ${\bf b}=(b_1,\cdots,b_s)$ we write ${\bf b} \subset {\bf a}$ if $r \ge s$ and
$b_i=a_i$ for any $i \le s.$

\begin{definition}
\label{def.2.1}
We say that an element $B$ of $\calh_n(\frko)^{\rm nd}$ is a weak canonical form if it satisfies the following conditions:
\begin{itemize}
\item[(1)] $B$ can be expressed as 
$$B=2^{k_1}J_1 \bot \cdots \bot 2^{k_r}J_r$$
with $k_1 \le \cdots \le k_r$, where $J_i \in S_{n_i}(\frko)_d \cap GL_{n_i}(\frko)$ with $1 \le n_i \le 2$  or $J_i \in \frac{1}{2}(S_2(\frko)_e \cap GL_2(\frko))$.
\item[(2)] If $J_i \in S_{n_i}(\frko)_d \cap GL_{n_i}(\frko)$, then $k_i <k_{i+1}$ and for any $j \le i-1$ such that $k_j=k_i$, $J_j$ belongs to  $\frac{1}{2}(S_2(\frko)_e \cap GL_2(\frko))$.
\item[(3)] Let  $J_{j}$ and $J_i$  be diagonal matrices and $J_{j-1},\ldots, J_{i+1} \in \frac{1}{2}(S_2(\frko)_e \cap GL_2(\frko))$ with $j >i$. Suppose that $k_j=k_i + 1$ and that  $\deg(2^{k_1} J_1 \bot \cdots \bot 2^{k_i}J_i)$ and $\mathrm{ord}( \det(2^{k_1} J_1 \bot \cdots \bot 2^{k_i}J_i))$ are even. Then $\xi_{2^{k_1} J_1 \bot \cdots \bot 2^{k_i}J_i}=0$.
\end{itemize}
\end{definition}
The notion of `weak canonical form' is essentially weaker than the `canonical form' in \cite{Wat} though a canonical form is not necessarily a weak canonical form. We say that an element $B$ of $\calh_n(\frko)^{\rm nd}$  has a weak canonical decomposition if there is a weak canonical form $B'$ such that $B \sim B'$.
By using the same argument as in the proof of the main result in \cite{Wat} combined with Theorem 2.4 of \cite{Cho}, we can prove that every $B \in \calh_n(\frko)^{\rm nd}$  has a weak canonical decomposition. See also Section \ref{sec:reduction} for a simple algorithm in the case $F=\QQ_2$.
\begin{definition}
\label{def.2.2}
Let  $B$ be an element of ${\mathcal H}_{n}(\frko)$ such that
$$B=2^{k_1}C_1 \bot \cdots \bot 2^{k_r}C_r $$
with $k_i \ge 0,$ where $C_i$ is a unimodular diagonal matrix or belongs to $\frac{1}{2}(S_2(\frko)_e \cap GL_2(\frko))$.
For a positive integer $j \le r$ put
$B^{[j]}=2^{k_1}C_1 \bot \cdots \bot 2^{k_j}C_j$.  Moreover, for a non-negative integer $m$ put 
$$\cald_m=\{ 1 \le j \le r \ | \ k_j=m \text{ and } C_j \text{ is diagonal} \},$$
and
$$\cale_m=\{ 1 \le j \le r \ | \ k_j=m \text{ and } C_j \in \frac{1}{2}(S_2(\frko)_e \cap GL_2(\frko)) \}.$$
We say that $B$ is pre-optimal if $B$ satisfies the following conditions:
\begin{itemize}
\item[(PO1)] For any non-negative integer $m$ we have
$$\sum_{j \in \cald_m} \deg C_j \le 2$$
and there is an integer $1 \le i \le j$ such that
$$\cale_m=\{i, i+1,\ldots,j-1\}.$$  
Here we make the convention that $\cale_m=\phi$ if $i=j$.
\item[(PO2)] Let $1 \le i < j \le r$.
\begin{itemize}
\item[(1)] Suppose that $i \in  \cald_{m_1}$ and $j \in \cald_{m_2}$. Then $m_1 \le m_2$.
\item[(2)]  Suppose that $i \in  \cale_{m_1}$ and $j \in \cale_{m_2}$. Then $m_1 \le m_2$.
\item[(3)]  Suppose that $i \in  \cald_{m_1}$ and $j \in \cale_{m_2}$. Then $m_1 \le m_2-1$.
\item[(4)]  Suppose that $i \in \cale_{m_1}$ and $j \in \cald_{m_2}$. Then $m_1 \le m_2+1$.
\end{itemize}
\item[(PO3)] If $C_i$ is a diagonal unimodular matrix of degree $2$, then, $i \ge 2$ and  one of the following conditions holds:
\begin{itemize}
\item[(1)] $\deg B^{[i]}$ is even and $\xi_{B^{[i-1]}}=\xi_{B^{[i]}}=0$.
\item[(2)] $\deg B^{[i]}$ is odd  and $\mathrm{ord}(\det B^{[i-1]})+k_i$ is even.
\end{itemize}
\item[(PO4)] Suppose that  $k_i=k_{i-1}-1$, $C_i$ is diagonal, and that  $C_{i-1} \in \frac{1}{2}(S_2(\frko)_e \cap GL_2(\frko))$.
Then $\deg C_i=2$, or $\deg C_i=1$ and   one of the following conditions holds:
\begin{itemize}
\item[(1)] $\deg B^{[i]}$ is even and $\mathrm{ord}(\det B^{[i]})$ is even
\item[(2)] $\deg B^{[i]}$ is odd  and $\xi_{B^{[i-1]}}=0$.
\end{itemize}
\item[(PO5)] Suppose that $C_i$ is diagonal,  $C_{i+1} \in \frac{1}{2} (S_2(\frko)_e \cap GL_2(\frko))$ and that $k_i = k_{i+1}-1$.
Then $\deg C_i=1$, and one of the following conditions holds:
\begin{itemize}
\item[(1)] $\deg B^{[i]}$ is even and $\mathrm{ord}(\det (B^{[i]}))$ is odd.
\item[(2)] $\deg B^{[i]}$ is odd and $\xi_{B^{[i-1]}}\not=0$ if $i \ge 2$.
\end{itemize}
\item[(PO6)] We have $\xi_{B^{[i]}}=0$ if $\deg B^{[i]}$ is even, $C_i$ and $C_{j}$ are unimodular diagonal, $C_{j-1},\ldots, C_{i+1} \in \frac{1}{2}(S_2(\frko)_e \cap GL_2(\frko))$ and $k_j=k_i+1.$ 
\end{itemize}
We call $(2^{k_1}C_1,\cdots,2^{k_r}C_r)$ the set of pre-optimal components of $B$ and denote it  by $\mathbb{POC}(B)$.
\end{definition}
%For $B \in \calh_n(\frko)^{\rm nd}$, we denote by $i(B)$ the least integer such that $2^{i(B)}B^{-1} \in \calh_n(\frko)$. 

\begin{proposition}
\label{prop.2.1}
Any  element $B$ of ${\mathcal H}_{n}(\frko)^{\rm nd}$ is $GL_n(\frko)$-equivalent to a  pre-optimal form.
\end{proposition}
\begin{proof}
We may suppose that $B$ is a weak canonical form:
\[2^{k_1}J_1 \bot \cdots \bot 2^{k_r} J_r\]
with $k_1 \le \cdots \le k_r$, and $J_i \in S_{n_i} \cap GL_{n_i}(\frko) \ (1 \le n_i \le 2)$ or $J_i \in \frac{1}{2}(S_2(\frko)_e \cap GL_2(\frko)) \ (i=1,\ldots,r)$. By induction on $r$, we prove that there is a 
pre-optimal form $\widetilde B$ which is equivalent to $B$ with $\mathbb{POC}(\widetilde B)=(2^{l_1}C_1,\ldots,2^{l_s}C_s)$ such that $l_s \le k_r$.
  Let $r=1$ and put $B=2^{k_1} J_r$. If $J_1 \in S_1(\frko) \cap GL_1(\frko)$ or $J_1 \in \frac{1}{2}(S_2(\frko)_e \cap GL_2(\frko))$, then $2^{k_1}J_1$ is itself a pre-optimal form with $\mathbb{POC}(B)=(2^{k_1}J_1)$. If $J_1=C_1 \bot C_2$ with $C_i \in 
S_1(\frko) \cap GL_1(\frko)$, then $2^{k_1}C_1 \bot 2^{k_1}C_2$ is a pre-optimal form with $\mathbb{POC}(2^{k_1}C_1 \bot 2^{k_1}C_2)=(2^{k_1}C_1,2^{k_1}C_2)$.

Let $r \ge 2$ and suppose that the assertion holds for any $B' \in \calh_{r'}(\frko)^{\rm nd}$ with $1 \le r' \le r-1$.

(1) Suppose that $\deg J_r =1$, and put $B_1=2^{k_1}J_1 \bot \cdots \bot  2^{k_{r-1}} J_{r-1}$. Then, by the induction assumption, there is a pre-optimal form $\widetilde B_1$ such that $\widetilde B_1$ is equivalent to $B_1$ with $\mathbb{POC}(\widetilde B_1)=(2^{l_1}C_1,\ldots,2^{l_t}C_t)$ such that $l_t \le k_{r-1}$. Put $\widetilde B=\widetilde B_1 \bot 2^{k_r}J_r$.  Then, $2^{k_r}J_r$ and $2^{l_t}C_t$ satisfy the conditions (PO1),(PO2) (1),(4), (PO4),(PO6). Hence $\widetilde B$ is a pre-optimal form with $\mathbb{POC}(\widetilde B)=(\mathbb{POC}(\widetilde B_1),2^{k_r}J_r)$ satisfying the required condition.

(2) Suppose that $J_r=u_1 \bot u_2$ with  $u_1,u_2 \in \frko^{\times}$, and put $B_1=2^{k_1}J_1 \bot \cdots \bot  2^{k_{r-1}} J_{r-1}$. Then, by the induction assumption, 
there is a pre-optimal form $\widetilde B_1$ such that $\widetilde B_1$ is equivalent to $B_1$ with $\mathbb{POC}(\widetilde B_1)=(2^{l_1}C_1,\ldots,2^{l_t}C_t)$ such that $l_t \le k_{r-1}$. 

(2.1) Suppose that either $n$ is even and $\xi_{B_1}=\xi_B=0$, or $n$ is odd and $\mathrm{ord}(\det B_1)+k_r$ is even.  Put $\widetilde B:=\widetilde B_1 \bot 2^{k_r}J_r$. Then, $2^{k_r}J_r$ and $2^{l_t}C_t$ satisfy the conditions (PO1),(PO2) (1),(4), (PO3),(PO4),(PO6). Hence $\widetilde B$ is a pre-optimal form with $\mathbb{POC}(\widetilde B)=(\mathbb{POC}(\widetilde B_1),2^{k_r}J_r)$ satisfying the required condition.

(2.2) Suppose that $B$ does not satisfy the condition in (2.1).  Put $\widetilde B:=\widetilde B_1 \bot 2^{k_r}u_1 \bot 2^{k_r}u_2$. Then, $2^{k_r}u_1,2^{k_r}u_2$ and $2^{l_t}C_t$ satisfy  the conditions (PO1),(PO2) (1),(4), (PO4),(PO6).
Hence $\widetilde B$ is a pre-optimal form with $\mathbb{POC}(\widetilde B)=(\mathbb{POC}(\widetilde B_1),2^{k_r}u_1, 2^{k_r}u_2)$ satisfying the required condition.

 (3) Suppose that $J_r=\bot_{i=1}^{m}  K_i$ with  $K_i \in \frac{1}{2} (S_2(\frko)_e \cap GL_2(\frko))$, and put
 $B_1=2^{k_1} J_1 \bot \cdots 2^{k_{r-m-1}}J_{r-m-1}$. 
Then, by the induction assumption,  there is a pre-optimal form $\widetilde B_1$ such that $\widetilde B_1$ is equivalent to $B_1$ with $\mathbb{POC}(\widetilde B_1)=(2^{l_1}C_1,\ldots,2^{l_t}C_t)$ such that $l_t \le k_{r-m}$. 

(3.1) Suppose that $k_r =k_{r-m}+1$ and $J_{r-m}=u$ with  $u \in \frko^{\times}$. Then $2^{l_t}C_t=2^{k_{r-m}}J_{r-m}$.

(3.1.1)  Suppose that  one of the following conditions holds
\begin{itemize}
\item[(a)] $\deg  B^{[r-m]}$ is even and $\mathrm{ord}(\det B^{[r-m]})$ is even.
\item[(b)] $r \ge m+2, \ \deg B^{[r-m-1]}$ is odd,  and $\xi_{B^{[r-m-1]}} =0$. 
\end{itemize}
Put $\widetilde B:=\widetilde B_1 \bot \bot_{i=1}^{m}  2^{k_r}K_i \bot 2^{k_{r-m}}J_{r-m}$. Then, $2^{k_{r-m}}J_{r-m}$
and $K_m$ satisfy the conditions (PO1),(PO2) (1),(4), (PO4),(PO6), and $2^{k_r}K_1$ and $2^{l_{t-1}}C_{t-1}$ satisfy
the conditions  (PO1),(PO2) (2),(3), (PO5). Hence $\widetilde B$ is a pre-optimal form with \\
$\mathbb{POC}(\widetilde B)=(\mathbb{POC}(\widetilde B_1), 2^{k_r}K_1,\ldots,2^{k_r}K_m, 2^{k_{r-m}}J_{r-m})$\\ satisfying the required condition.
 
(3.1.2) Suppose that $B$ does not satisfy the condition in (3.1.1).  Then, in a way similar to above, we can prove that\\
 $\widetilde B=\widetilde B_1 \bot 2^{k_{r-m}} J_{r-m} \bot  \bot_{i=1}^{m}  2^{k_r}K_i$  is a pre-optimal form with \\
$\mathbb{POC}(\widetilde B)=(\mathbb{POC}(\widetilde B_1), 2^{k_{r-m}}J_{r-m},2^{k_r}K_1,\ldots,2^{k_r}K_m)$\\ satisfying the required condition.

(3.2) Suppose that $k_r =k_{r-m}+1, \ J_{r-m}=u_1 \bot u_2$ with  $u_1,u_2 \in \frko^{\times}$.

(3.2.1)  Suppose that one of the following conditions holds
\begin{itemize}
\item[(a)] $r=m+1$.
\item[(b)] $r \ge m+2$, $\deg B^{[r-m-1]}$ is even, and $\xi_{B^{[r-m-1]}} \not=0$.
\item[(c)] $\deg B^{[r-m-1]}$ is odd, and  $\mathrm{ord}(\det (B^{[r-m-1]})+ k_{r-m}$ is odd.
\end{itemize}
Put $\widetilde B:=\widetilde B_1  \bot  2^{k_{r-m}}u_1 \bot \bot_{i=1}^{m}  2^{k_r}K_i  \bot 2^{k_{r-m}}u_2$. Then 
$2^{k_{r-m}}u_2$ and $2^{k_r}K_m$ satisfy the conditions (PO1),(PO2) (1),(4),(PO4),(PO6), and  $2^{k_r}K_1$ and 
$2^{k_{r-m}}u_2$ satisfy the conditions (PO1),(PO2) (2),(3), (PO5). Hence $\widetilde B$ 
 is a pre-optimal form with \\
$\mathbb{POC}(\widetilde B)=(\mathbb{POC}(\widetilde B_1),2^{k_{r-m}}u_1, 2^{k_r}K_1,\ldots,2^{k_r}K_m, 2^{k_{r-m}}u_2)$ satisfying the required condition.

(3.2.2) Suppose that one of the following conditions holds
\begin{itemize}
\item[(a)] $r \ge m+2$, $\deg B^{[r-m-1]}$ is even, and $\xi_{B^{[r-m-1]}} =\xi_{B^{[r-m]}}=0$.
\item[(b)] $\deg B^{[r-m-1]}$ is odd, and  $\mathrm{ord}(\det (B^{[r-m-1]})+ k_{r-m}$ is even.
\end{itemize}
Put $\widetilde B:=\widetilde B_1  \bot  \bot_{i=1}^{m}  2^{k_r}K_i \bot 2^{k_{r-m}}J_{r-m}$. Then, $2^{k_{r-m}}J_{r-m}$ and $2^{k_r}K_m$ satisfy the conditions (PO1),(PO2) (1),(4), (PO3),(PO6). Hence $\widetilde B$ is a pre-optimal form with $\mathbb{POC}(\widetilde B)=(\mathbb{POC}(\widetilde B_1), 2^{k_r}K_1,\ldots,2^{k_r}K_m, 2^{k_{r-m}}J_{r-m})$\\ satisfying the required condition.

(3.2.3) Suppose that $r \ge m+2$, $\deg B^{[r-m-1]}$ is even, $\xi_{B^{[r-m-1]}} =0$, and $\xi_{B^{[r-m]}}\not=0$.
Then, in a way similar to above, we can prove that   
$\widetilde B=\widetilde B_1 \bot \bot_{i=1}^{m}  2^{k_r}K_i  \bot  2^{k_{r-m}}u_1 \bot 2^{k_{r-m}}u_2$ is a pre-optimal form with \\
$\mathbb{POC}(\widetilde B)=(\mathbb{POC}(\widetilde B_1), 2^{k_r}K_1,\ldots,2^{k_r}K_m, 2^{k_{r-m}}u_1,2^{k_{r-m}}u_2)$\\ satisfying the required condition.

(3.3)  Suppose that $k_r \ge k_{r-m}+2$, or that
$k_r=k_{r-m}+1$ and $J_{r-m} \in \frac{1}{2}(S_2(\frko)_e \cap GL_2(\frko))$. 
Put $\widetilde B=\widetilde B_1 \bot \bot_{i=1}^{m}  2^{k_r}K_i$. Then $2^{k_r}K_1$ and $2^{l_t}C_t$ satisfy the conditions (PO1),(PO2) (2),(3), (PO5). Hence $\widetilde B$  is a pre-optimal form with  $\mathbb{POC}(\widetilde B)=(\mathbb{POC}(\widetilde B_1), 2^{k_r}K_1,\ldots,2^{k_r}K_m)$\\ satisfying the required condition.

\end{proof}

Let
$$B=2^{k_1}C_1 \bot \cdots \bot 2^{k_r}C_r $$
be a pre-optimal form, and put $B^{[j]}=2^{k_1}C_1 \bot \cdots \bot 2^{k_j}C_j$ as stated above. For $i$ let $l(i)$ be the least integer such that $i \le \deg B^{[l(i)]}.$

\begin{example}
\label{exp.2.1}  From now on for $i=1,2,3,4$ let $u_i \in \frko^{\times}$ and $K_i \in \frac{1}{2}(GL_2(\frko) \cap S_2(\frko)_e)$.\\
(I) The following (1) $\sim$ (6) are weak canonical forms.
\begin{itemize}
\item[(1)] $B=2^{k_1}K_1 \bot 2^{k_2}K_2$ with $k_1 \le k_2$ is a pre-optimal form and $\mathbb{POC}(B)=(2^{k_1}K_1, 2^{k_2}K_2)$.
\item[(2)] $B=2^{k_1}K_1 \bot 2^{k_2}u_2 \bot 2^{k_3}u_3$ with $k_1 \le k_2 \le k_3$ is a pre-optimal form and $\mathbb{POC}(B)=(2^{k_1}K_1, 2^{k_2}u_2,2^{k_2}u_3)$.
\item[(3)] $B=2^{k_1}u_1 \bot 2^{k_2}K_2 \bot 2^{k_3}u_3$ with $k_1 < k_2 \le k_3$ is a pre-optimal form and $\mathbb{POC}(B)=(2^{k_1}u_1, 2^{k_2}K_2,2^{k_3}u_3)$.
\item[(4)] $B=2^{k_1}u_1 \bot 2^{k_2}u_2 \bot 2^{k_3}K_3$ with $k_1 \le k_2 < k_3$. \\
\begin{itemize}
\item[(4.1)] Suppose that $k_3 \ge k_2+2$ or $k_3=k_2+1$ and $k_1+k_2$ is odd. Then, $B$ is pre-optimal and
 $\mathbb{POC}(B)=(2^{k_1}u_1, 2^{k_2}u_2,2^{k_3}K_3)$.
\item[(4.2)] Suppose that  $k_3=k_2+1$ and $k_1+k_2$ is even. Then, $B$ is not pre-optimal.
 But $B'=2^{k_1}u_1 \bot  2^{k_3}K_3 \bot 2^{k_2}u_2$ is equivalent to $B$ and pre-optimal, and 
 $\mathbb{POC}(B')= (2^{k_1}u_1, 2^{k_3}K_3, 2^{k_2}u_2)$.
\end{itemize}
\item[(5)] $B=2^{k_1}u_1 \bot 2^{k_2}u_2 \bot 2^{k_3}u_3 \bot 2^{k_4}u_4$ with $k_1 \le k_2, \ k_2+2 \le  k_3 \le k_4 $ is pre-optimal. \\
\begin{itemize}
\item[(5.1)] Suppose that $k_3=k_4$ and $\xi_{B^{(2)}}=\xi_B=0$. Then  
 $\mathbb{POC}(B)=(2^{k_1}u_1, 2^{k_2}u_2,2^{k_3} u_3 \bot 2^{k_3}u_4)$.
\item[(5.2)] Suppose that  $B$ does not satisfy the condition in (5.1). Then, 
 $\mathbb{POC}(B)= (2^{k_1}u_1, 2^{k_2}u_2, 2^{k_3}, 2^{k_4}u_4)$.
\end{itemize}
\item[(6)] $B=2^{k_1}u_1 \bot 2^{k_2}u_2 \bot 2^{k_3}u_3 \bot 2^{k_4}u_4$ with $k_1 \le k_2= k_3 < k_4 $ is pre-optimal. \\
\begin{itemize}
\item[(6.1)] Suppose that $k_1+k_2$ is even.  Then  
 $\mathbb{POC}(B)=(2^{k_1}u_1, 2^{k_2}u_2 \bot 2^{k_3} u_3,2^{k_4}u_4)$.
\item[(6.2)] Suppose that  $k_1+k_2$ is odd.  Then, 
 $\mathbb{POC}(B)= (2^{k_1}u_1, 2^{k_2}u_2, 2^{k_3}u_3, 2^{k_4}u_4)$.
\end{itemize}
\end{itemize}
(II) Let $B=2^{k_1}u_1 \bot 2^{k_2}u_2 \bot 2^{k_3}u_3 \bot 2^{k_4}u_4$ with $k_1 \le k_2, \ k_2+1 = k_3 \le k_4 $
\begin{itemize}
\item[(II.1)] Suppose that $\xi_{B^{(2)}} =0$. Then $B$ is weak-canonical and pre-optimal, and the same property  as (5) holds.
\item[(II.2)] Suppose that $\xi_{B^{(2)}} \not=0$. Then $B$ is neither weak-canonical nor  pre-optimal. But there are $u_1',u_2' \in \frko^{\times}$ such that $2^{k_1}u_1' \bot 2^{k_2}u_2' \bot 2^{k_3}u_3 \bot 2^{k_4}u_4$ is equivalent to $B$ and satisfies the same condition as (II.1).
\end{itemize}
\end{example}

\begin{lemma} 
\label{lem.2.1}
 Let 
$$B=2^{k_1}C_1 \bot \cdots \bot 2^{k_r}C_r $$
be a pre-optimal form in ${\mathcal H}_n(\frko)$ as above. Let 
$$B_1=2^{k_1}C_1 \bot \cdots \bot 2^{k_{r-1}}C_{r-1}$$
and
$$B_2= 2^{k_r}C_r.$$
Then we have  the following. 
\begin{itemize}
\item[(1)] We have
$$(\GK(B)_1,\cdots,\GK(B)_{n_1}) \preceq \GK(B_1).$$
\item[(2)] If there are integers $b_{n_1+1},\cdots,b_{n_1+n_2}$ such that \\
$(\GK(B_1),b_{n_1+1},\cdots,b_{n_1+n_2}) \in \bfS(\{B\}),$
then 
$$\GK(B_1)=(\GK(B)_1,\cdots,\GK(B)_{n_1}).$$

\end{itemize}
\end{lemma}

\begin{proof}
The first assertion follows from \cite[Lemma 1.2]{IK1}. 
Suppose that there are integers $b_{n_1},\cdots,b_n$ 
and $U \in GL_n(\frko)$ such that $(\GK(B_1),b_{n_1},\cdots,b_n) \in S(B[U]).$
Then we have  $(\GK(B_1),b_{n_1+1},\cdots,b_{n}) \preceq (a_1,\cdots,a_n),$ and in particular $\GK(B_1) \preceq (a_1,\cdots,a_{n_1}).$
This proves the assertion. 
\end{proof}

\begin{theorem}
\label{th.2.1} Let
 $$2^{k_1}C_1 \bot \cdots \bot 2^{k_r}C_r$$
be a pre-optimal form in $\calh_n(\frko)$, and $B_1 =  2^{k_1} C_1 \bot \cdots \bot 2^{k_{r-1}}C_{r-1}$. 
Put $n_1=\deg B_1$, and  $m=k_j$ or $m=k_j-2$ according as $C_j$ is unimodular diagonal or not. Suppose that the following conditions hold
\begin{itemize}
\item[(1)] There are integers $b_{n_1+1},\ldots,b_n$ such that 
$(\GK(B_1),b_{n_1+1},\ldots,b_n) \in \bfS(\{B\})$
\item[(2)] $a_{n_1+1} \le m +2$. 
\end{itemize}
Then there is an optimal  basis $\{u_i \}_{1 \le i \le n}$ of  $L_B$ such that
$\{u_i \}_{1 \le i \le n_1}$ is an optimal basis of $L_{B_1}$.
\end{theorem}

\bigskip

\begin{proof}
Let  $\GK(B)=(a_1,\cdots,a_n)$. Then by the assumption (1) and Lemma \ref{lem.2.1}, we have $\GK(B_1)=(a_1,\ldots,a_{n_1})$. 
We note that we have $a_{n_1} \le m+2$  by the assumption (2). 
Let 
$\{\psi_1,\cdots,\psi_n\}$ be  an optimal  basis of $L_B.$ 
First suppose that  
$B_2=2^k \epsilon.$ Then $m=k.$
Let $\widetilde \phi_n$ be a basis of $L_{B_2}.$ 
For $i=1,\cdots,n,$ write $\psi_i=\phi_i +  c_{i}\widetilde \phi_{n}$ with $c_i \in \frko,$ where
$\phi_i$ is an element $L_{B_1}.$ Then there is an integer $1 \le i \le n$ such that $c_i \in \frko^*.$ Take the greatest integer $i_0$ satisfying 
such a condition, and for $1 \le i \le n$  such that $i \not=i_0$ put 
$$\phi_i'=\begin{cases}
\phi_i-c_{i_0}^{-1}c_i\phi_{i_0} & \text{ if } i \le i_0-1 \\
\phi_{i} & \text{ if }  i \ge i_0+1,
\end{cases}$$
and $\Phi'=\{ \phi_i ' \ (1 \le i \le n, i \not= i_0) \}$.  Then $\Phi'$ forms a basis of $L_{B_1},$ and can be expressed as
$$\phi_i'=\psi_i+d_i\psi_{i_0}+2e_i\widetilde \phi_n$$
with some $d_i,e_i \in \frko$ such that $d_i=0$ if $i > i_0.$
 We note that
\begin{align*}
&(\phi_i',\phi_j')+4e_ie_j(\widetilde \phi_n,\widetilde \phi_n) \\
&=(\psi_i,\psi_j)+d_i(\psi_{i_0},\psi_j)+d_j(\psi_{i_0},\psi_i)+d_id_j(\psi_{i_0},\psi_{i_0})
\end{align*}
and 
$${\rm ord}(4e_ie_j(\widetilde \phi_n,\widetilde \phi_n)) \ge m+3$$
for any $1 \le i,j \le n$ such that $i \not=i_0, j \not=i_0.$
Hence we have 
$${\rm ord}(2^{-\delta_{i,j}}(\phi_i',\phi_j')) \ge (a_i+a_j)/2$$
 for any $1 \le i,j \le n$ such that $i \not=i_0$ and $j \not=i_0$. 
This implies that the sequence $(b_1,\ldots,b_{n-1})$ defined by
$$b_i=\begin{cases} a_i & \text{ if } 1 \le i \le i_0-1\\
a_{i+1} & \text{ if } i_0 \le i \le n-1
\end{cases}$$
belongs to $S(\Phi')$, and hence we have $(b_1,\cdots,b_{n-1}) =\GK(B_1)$
remarking that $(a_1,\ldots,a_{n-1}) \preceq (b_1,\ldots,b_{n-1})$.
Put $\Phi=\Phi' \cup \{\psi_n \}$. Then $\Phi$ forms a basis of $L_B$, and  since we have  $a_n \le k_r+2$,  
we easily see that 
$$\mathrm{ord}((\phi_i',\psi_n)) \ge (a_n+a_i)/2$$
for any $1 \le i \le n$ such that $i \not=i_0$.
Hence we have $(a_1,\ldots,a_n) \in S(\Phi)$, and $\Phi$ is an optimal basis of $L_B$.
This implies that $\Phi$ satisfies the required property.

Suppose that $B_2=2^kC$ with $\deg C=2.$ Then $m=k$ or $m=k-2$ according as $C$ is  unimodular diagonal or not.
Let $\widetilde \phi_{n1},\widetilde \phi_{n2}$ be a basis of $L_{B_2}$. 
Then by using the same argument as above, we can show that 
there are two integers $1 \le i_1 <i_2 \le n$ and a basis $\phi_{i}' \ (i \not=i_1,i_2)$  of $L_{B_1}$ such that
$$\phi_i'=\psi_i+d_{i1}\psi_{i_1}+d_{i2}\psi_{i_2}+2e_{i1}\widetilde \phi_{n1}+2e_{i2}\widetilde \phi_{n2},$$
where $d_{ij} ,e_{ij} \in \frko$ such that $d_{i1}=0$ if $i >i_1$ and $d_{i2}=0$ if $i >i_2.$ 
Put $\Phi'=\{ \phi_i' \ (1 \le i \le n, \ i \not=i_1,i_2)\} $ and 
$$b_i=\begin{cases}
a_i & \text{ if } 1 \le i \le i_1-1 \\
a_{i+1} & \text{ if } i_1 \le i \le i_2-1 \\
a_{i+2} & \text{ if } i_2 \le i \le n-1.
\end{cases}$$
We note that
\begin{align*}
&(\phi_i',\phi_j')+4\{e_{i1}e_{j1}(\widetilde \phi_{n1},\widetilde \phi_{n1})+(e_{i1}e_{j2}+e_{i2}e_{j1})(\widetilde \phi_{n1},\widetilde \phi_{n2})+e_{i2}e_{j2}(\widetilde \phi_{n},\widetilde \phi_{n2})\}\\
&=(\psi_i,\psi_j)+d_{i1}(\psi_{i_1},\psi_j)+d_{i2}(\psi_{i_2},\psi_j)+ d_{j1}(\psi_{i_1},\psi_i)+d_{j2}(\psi_{i_2},\psi_i)\\
&+d_{i1}d_{j1}(\psi_{i_1},\psi_{i_1})
+(d_{i1}d_{j2}+d_{j1}d_{i2})(\psi_{i_1},\psi_{i_2})+d_{i2}d_{j2}(\psi_{i_2},\psi_{i_2})
\end{align*}
and
$${\rm ord}(4\{e_{i1}e_{j1}(\widetilde \phi_{n1},\widetilde \phi_{n1})+(e_{i1}e_{j2}+e_{i2}e_{j1})(\widetilde \phi_{n1},\widetilde \phi_{n2})+e_{i2}e_{j2}(\widetilde \phi_{n},\widetilde \phi_{n2})\}) \ge m+3$$
for any $1 \le i,j \le n$ such that $i \not=i_1, i_2, j \not=i_1,i_2.$
Put $\Phi=\Phi' \cup \{\psi_{n-1},\psi_n \}$.
Then similarly to above we can show that 
$${\rm ord}(2^{-\delta_{i,j}}(\phi_i',\phi_j')) \ge (a_i+a_j)/2$$
 for any $1 \le i,j \le n$ such that $i \not=i_1,i_2, j \not=i_1,i_2$,
$$\mathrm{ord}((\phi_i',\psi_j)) \ge (a_i+a_j)/2$$
for any $1 \le i \le n-2, n-1 \le j \le $, and
$$\mathrm{ord}(2^{-\delta_{ij}}(\psi_i,\psi_j) \ge (a_i+a_j)/2$$
for any $n-1 \le i,j \le n$.
Thus, by using the same argument as above, we can prove the assertion.

\end{proof}

%\bigskip

\bigskip
Let $B \in {\mathcal H}_n(\frko).$
Let $n$ be odd. Then we  recall that 
$$\Delta(B)=\mathrm{ord}(\det B)+n-1.$$
Let $n$ be even. Then we remark that 
$$\Delta(B)=\left\{\begin{array}{ll}
\mathrm{ord}(\det B)+n-2 & \ {\rm if} \ \mathrm{ord}(\det B) \equiv 1 \ {\rm mod} \ 2 \\
\mathrm{ord}(\det B)+n-1 & \ {\rm if} \ \mathrm{ord}(\det B) \equiv 0 \ {\rm mod} \ 2 \ {\rm and } \ \xi_B=0 \\
\mathrm{ord}(\det B)+n & \ {\rm if} \  \ \xi_B\not=0.
\end{array}
\right.$$
Hence, we easily obtain:

\bigskip

\begin{lemma}
\label{lem.2.2}
 Let $B=2^{k_1}C_1 \bot \cdots \bot 2^{k_{r}}C_{r}$ 
be a pre-optimal form in ${\mathcal H}_n(\frko).$
Put $B_1 = 2^{k_1}C_1 \bot \cdots \bot 2^{k_{r-1}}C_{r-1},$ 
and $B_2 = 2^{k_r}C_r.$ 
\begin{itemize}
\item [(1)] Let  $n_2=1.$
\begin{itemize}
\item [(1.1)] Let $n_1$ be even. Then
$$\Delta(B)-\Delta(B_1)=
\left\{ \begin{array}{ll}
k_r+2  & \ {\rm if} \ \mathrm{ord}(\det B_1) \ {\rm is \ odd}\\
k_r +1   &    \ {\rm if} \mathrm{ord}(\det B_1) \ {\rm is \ even \ and } \ \xi_{B_1}=0\\
k_r & \ {\rm if} \ \xi_{B_1} \not=0.
\end{array} \right.$$
\item[(1.2)] Let $n_1$ be odd. Then
$$\Delta(B)-\Delta(B_1)=
\left\{ \begin{array}{ll}
k_r  & \ {\rm if} \ \mathrm{ord}(\det B) \ {\rm is \ odd}\\
k_r +1   &    \ {\rm if} \ \mathrm{ord}(\det B) \ {\rm is \ even \ and } \ \xi_{B}=0\\
k_r+2 & \ {\rm if} \ \xi_{B} \not=0.
\end{array} \right.$$
\end{itemize}
\item [(2)] Let $n_2=2$ and $C_r$ is unimodular diagonal. 
Then
$$\Delta(B)=\Delta(B_1)+2k_r+2.$$
\item [(3)] Let $C_r=H$ or $Y.$ Then 
$$\Delta(B)=\Delta(B_1)+2k_r.$$
\end{itemize}
\end{lemma}

\section{Explicit formula for GK invariant}

\begin{theorem}
\label{th.3.1}
Let $r \ge 2$. 
Let $B=2^{k_1}C_1 \bot \cdots \bot 2^{k_r}C_r$ be a pre-optimal form of degree $n$, and let $\GK(B)=(a_1,\ldots,a_n)$. For each $1 \le s \le r$ put
$\widetilde n_s=\deg C_1+\cdots +\deg C_s$. Then, for any $1 \le s \le r$  we have the following:
\begin{itemize}
\item[(1)] Let $C_s \in \frac{1}{2}(S_2(\frko)_e \cap GL_2(\frko))$. Then 
\[(a_{\widetilde n_s-1},a_{\widetilde n_s})=(k_s,k_s).\]
\item[(2)] Let $C_s$ be a unimodular diagonal matrix of degree $1$. 
\begin{itemize}
\item[(2.1)] Suppose that $\widetilde n_s$ is odd. Then
\[a_{\widetilde n_s}=
\begin{cases} k_s+2 & \text{ if } \mathrm{ord}(\det B^{[s-1]}) \text{ is odd} \\
k_s+1 &   \text{ if } \mathrm{ord}(\det B^{[s-1]}) \text{ is even and } \xi_{B^{[s-1]}}=0 \\
k_s &   \text{ if }  \xi_{B^{[s-1]}}\not=0.
\end{cases}
\]
\item[(2.2)] Suppose that $\widetilde n_s$ is even. Then
\[a_{\widetilde n_s}=
\begin{cases} k_s & \text{ if } \mathrm{ord}(\det B^{[s]}) \text{ is odd} \\
k_s+1 &  \text{ if } \mathrm{ord}(\det B^{[s]}) \text{ is even and } \xi_{B^{[s]}}=0 \\
k_s +2 &  \text{ if }  \xi_{B^{[s]}}\not=0.
\end{cases}
\]
\end{itemize}
\item[(3)]  Let $C_s$ be a unimodular diagonal matrix of degree $2$.
Then
\[(a_{\widetilde n_s-1},a_{\widetilde n_s})=(k_s+1,k_s+1)\]
\end{itemize}
\end{theorem}

To prove the above theorem, we need some preliminary results.
\begin{lemma}
\label{lem.3.1} Let $n, m$ and $k$ be non-negative integers such that $1 \le m \le 2$ and $n > m$. Let $B_1$ be a reduced form of degree $n-m$ with $\GK$ type $(\ua',\sigma')$ and $C$  a diagonal unimodular matrix of degree $m$, and put $B=B_1 \bot 2^{k}C$.  Assume that
there is an integer $j_0 \in \calp^0(\sigma')$ such that $\mathrm{ord}(b_{j_0,j_0}) \equiv k \text{ mod } 2$. 
Then there 
is a matrix $B'$ which is equivalent to $B$ such that 
$(B')^{(n-m)}=B^{(n-m)}$,
and $(\ua',k+1,k+1) \in S(B')$.
\end{lemma} 
\begin{proof} Put $\ua'=(a_1,\ldots,a_{n-m})$ and $b_{j_0,j_0}=2^{k_0}u_{j_0}$ with $u_{j_0} \in \frko^{\times}$.  First let $C=u_{n-1} \bot u_{n}$ with $u_{n-1},u_{n} \in \frko^{\times}$.
We can take  elements $v_{j_0,n-1}$ and $v_{n-1,n-1}$ of $\frko^{\times}$ such that
$u_{j_0}v_{j_0,n-1}^2-u_{n-1} \in 2\frko$ and $u_{n-1}v_{n-1,n}^2-u_{n} \in 2\frko$.
Put 
$$B'=(b_{ij}')=B \left[\left(\begin{smallmatrix}  1 & 0  & 0      & 0           & 0          & 0 \\
                                                0 & 1 & 0      &0            & 2^{(k-k_{0})/2}v_{j_0,n-1} & 0 \\
                                                0 & 0 & 1      &0            & 0          & 0 \\
                                                0 & 0 & \ddots &             & \vdots & \vdots\\
                                                0 & 0 & 0 & \hdots        &  1          &  v_{n-1,n} \\
                                                0 & 0 & 0 &\hdots         &  0          &  1 
\end{smallmatrix}\right)\right].$$
Then, we have $(B')^{(n-m)}=B^{(n-m)},$ and 
$$\mathrm{ord}(2b_{i,j}') \ge (a_i +k+1)/2 \text{ for any } 1 \le i \le n-m, n-m \le j \le n,$$
and
$$\mathrm{ord}(2^{1-\delta_{ij}}b_{ij}') \ge k+1 \text{ for any } n-m+1 \le i,j \le n.$$
Namely, $B'$ satisfies the required conditions.
Similarly, the assertion holds in the case $\deg C_r=1$.

\end{proof}

\begin{proposition}
\label{prop.3.1}
 Let $r \ge 2$.
Let $B=2^{k_1}C_1 \bot \cdots \bot 2^{k_r}C_r$ be a pre-optimal form of degree $n$,  and $B_1=2^{k_1}C_1 \bot \cdots \bot 2^{k_{r-1}}C_{r-1}$. Put $n_r=\deg C_r$ and $\GK(B_1)=(b_1,\ldots,b_{n-n_r})$. Suppose that $b_1,\ldots,b_{n-n_r}$ satisfy  the condition in Theorem \ref{th.3.1}.
Then there are integers $c_{n-n_r+1},\ldots,c_n$ such that
$$(\GK(B_1),c_{n-n_r+1},\ldots,c_n) \in \bfS(\{B\}).$$
\end{proposition}
\begin{proof}
Let $\widetilde B_1$ be a reduced  form with $\GK$ type $(\GK(B_1),\sigma_1)$ such that $\widetilde B_1 \sim B_1$ and put $\widetilde B=\widetilde B_1 \bot 2^{k_r}C_r$.
Put ${\bf k}_r=(k_r)$ or ${\bf k}_r=(k_r,k_r)$ according as $n_r=1$ or $2$. We divide the proof into several cases.\\
(1) Suppose that $k_r \ge k_{r-1}+2$. Then, by the assumption, $b_{n-n_r} \le k_r$, and $(\GK(B_1),{\bf k}_r) \in S(\widetilde B)$.\\
(2) Suppose that $C_{r-1} \in \frac{1}{2}(S_2(\frko)_e \cap GL_2(\frko))$ and $k_{r-1} \le k_r$.  Then, by the assumption $b_{n-n_r}=k_{r-1} \le k_r$.
Hence $(\GK(B_1),{\bf k}_r) \in S(\widetilde B)$.\\
(3) Suppose that $k_r=k_{r-1}+1$,  $C_{r-1}$ is a diagonal  matrix, and that 
$C_r \in \frac{1}{2}(S_2(\frko)_e \cap GL_2(\frko))$. 
Then, by (PO5), and by the assumption, $b_{n-2}=k_{r-1}$, and hence
$(\GK(B_1),k_r,k_r) \in S(\widetilde B)$.\\
(4) Suppose that $k_r=k_{r-1}+1$ and that $C_{r-1}$ and $C_r$ are diagonal. 
First suppose that $\deg C_r=2$. If $\deg B_1$ is even, then $\xi_{B_1}=\xi_B=0$ by (PO3) (1). Hence, by the assumption $b_{n-2} \le k_{r-1}+1$, and $(\GK(B_1),k_r,k_r) \in S(\widetilde B)$. If $\deg B_1$ is odd, then $\mathrm{ord}(\det B_1)+k_r$ is even. Hence, by Proposition \ref{prop.1.1}, there is an integer $j_0 \in \calp^0(\sigma_1)$ such that  $\mathrm{ord}(b_{j_0,j_0}) \equiv k_r \text{ mod } 2$.  Then, by Lemma \ref{lem.3.1}, there is an element $\widetilde B' \in \calh_n(\frko)$ such that
$\widetilde B' \sim \widetilde B$ and $(\GK(B_1),k_r+1,k_r+1) \in S(\widetilde B')$. Next suppose that $\deg C_r=1$. If $\deg B_1$ is even, then $\xi_{B_1}=0$, and $b_{n-1} \le k_{r-1}+1$. Hence, $(\GK(B_1),k_r) \in S(\widetilde B)$. Suppose that $\deg B_1$ is odd. If $b_{n-1} \le k_{r-1}+1$, then 
$(\GK(B_1),k_r) \in S(\widetilde B)$. If $b_{n-1}=k_{r-1}+2$, then $r \ge 3$,  $\deg C_{r-1}=1$ and $\mathrm{ord}(\det B^{(n-2)})$ is odd. This implies that $\mathrm{ord}(\det B_1)+k_r$ is even. 
Then, by Proposition \ref{prop.1.1}, there is an integer $j_0 \in \calp^0(\sigma_1)$ such that $\mathrm{ord}(b_{j_0,j_0})  \equiv k_r \text{ mod } 2$. Then, by Lemma \ref{lem.3.1}, 
there is an element $B' \in \calh_n(\frko)$ such that
$\widetilde B' \sim \widetilde B$ and $(\GK(B_1),k_r+1) \in S(\widetilde B')$.\\
(5) Suppose that $k_r=k_{r-1}$ and that $\deg C_{r-1}=\deg C_r=1$. 
If $r=2$, then the assertion holds. Suppose that $r \ge 3$. 
If $\deg B^{(n-2)}$ is even, then $\mathrm{ord}(\det B^{(n-2)})$ 
is even by (PO2). Then $b_{n-1} \le k_{r-1}+1$ and using the same argument as above, we can prove that  there is a matrix $\widetilde B'$ which is equivalent to $\widetilde B$ such that $(\GK(B_1),k_r+1) \in S(\widetilde B')$. If $\deg B^{(n-2)}$ is odd, then $\mathrm{ord}(\det B_1)$ is odd again by (PO2). Then $b_{n-1}=k_r$ and $(\GK(B_1),k_r) \in S(\widetilde B)$. \\
(6) Suppose that $k_r=k_{r-1}-1$. Then $r \ge 3$, 
$C_{r-1} \in \frac{1}{2} (S_2(\frko)_e \cap GL_2(\frko))$ and $C_r$ is diagonal by (PO3). Moreover, $b_{n-n_r}=k_{r-1}$.
First suppose that $\deg C_r =1$. 
If $\deg B$ is even, then $\mathrm{ord}(\det B)$ is even by (PO3) (1).
Then, there is an integer $1 \le j_0 \le n-3$ such that
$b_{j_0,j_0} \not=0$ and $\mathrm{ord}(b_{j_0,j_0}) \equiv k_r \text{ mod } 2$. Hence, by Lemma \ref{lem.3.1}, there is a matrix $\widetilde B'$ such that $\widetilde B'$ is equivalent to $\widetilde B$ and $(\GK(B_1),k_r+1) \in S(\widetilde B')$. \\
If $\deg B$ is odd, then  $\xi_{B_1}=0$. Then, by Proposition \ref{prop.1.1} and  Lemma \ref{lem.3.1}, there is a matrix $\widetilde B'$ such that $\widetilde B' \sim \widetilde B$ and $(\GK(B_1),k_r+1) \in S(\widetilde B')$.
Next suppose that $\deg C_r =2$. If $\deg B$ is even, then $\mathrm{ord}(\det B)$ is even. If $\deg B$ is odd, then $\mathrm{ord}(\det B_1)+k_r$ is even. In any case, by using the same argument as above we can prove the assertion.

\end{proof}

{\bf Proof of Theorem \ref{th.3.1}} It suffices to prove the assertion  for the case $s=r$. Clearly the assertion holds for $r=1$. Let $r>1$ and suppose that the assertion holds for $r-1$. Then, by Proposition \ref{prop.3.1} and Lemma \ref{lem.2.2}, we have $\GK(B^{[r-1]})=(a_1,\ldots,a_{n-n_r})$ with $n_r=\deg C_r$. Then the assertion follows from Lemma \ref{lem.2.2} and Theorem \ref{th.1.1}.

\section{Explicit formula for a naive EGK datum}\label{sec:negk}

First we introduce some definitions.
Put ${\mathcal Z}_3=\{0,1,-1 \}$. 
\begin{definition} 
\label{def:4.1}
An element $H=(a_1,\ldots,a_n;\vep_1,\ldots,\vep_n)$
of $\ZZ_{\ge 0}^n \times {\mathcal Z}_3^n$ is said to be a naive  EGK datum of length $n$  if the following conditions hold:
\begin{itemize}
\item [(N1)] $a_1 \le \cdots \le a_n$.
\item [(N2)] Suppose that $i$ is even. 
Then $\vep_i \not=0$ if and only if $a_1+\cdots+a_i$ is even.
\item [(N3)] If $i$ is odd, then $\vep_i \not=0$. 
\item [(N4)] $\vep_1=1$.
\item [(N5)] If $i  \ge 3$ is odd and $a_1+\cdots + a_{i-1}$ is even, then $\vep_i=\vep_{i-2}\vep_{i-1}^{a_i+a_{i-1}}$.
\end{itemize}
We denote the set of naive EGK data of length $n$ by $\mathcal{NEGK}_n$.
\end{definition}

\begin{definition} %Definition 4.2 
\label{def:4.2}
Let $G=(n_1,\ldots, n_r; m_1, \ldots, m_r; \zeta_1, \ldots, \zeta_r)$ be an element of $\ZZ_{>0}^r \times \ZZ_{\ge 0}^r  \times {\mathcal Z}_3^r$.
Put $n^\ast_s=\sum_{i=1}^s n_i$ for $s \leq r$.
We say that $G$ is an EGK datum of length $n$  if the following conditions hold:
\begin{itemize}
\item [(E1)] $n^\ast_r=n$ and $m_1 <\cdots < m_r$.
\item[(E2)] Suppose that $n^\ast_s$ is even. 
Then $\zeta_s \not=0$ if and only if $m_1 n_1+\cdots+m_s n_s$ is even.
\item [(E3)] Suppose that $n^\ast_s$ is odd. 
Then $\zeta_s \not=0$. 
Moreover, we have
\begin{itemize}
\item[(a)] Suppose that $n^\ast_i$ is even for any $i<s$.
Then we have
\[
\zeta_s
=\zeta_1^{m_1+m_2}
\zeta_2^{m_2+m_3}
\cdots
\zeta_{s-1}^{m_{s-1}+m_s}.
\]
In particular, $\zeta_1=1$ if $n_1$ is odd.
\item[(b)]
%Suppose that both $n_s$ and $m_1+\cdots +(n_s-1) m_s$ are even.
%Then $\zeta_s=\zeta_{s-1}$.
%\item[(3)] 
Suppose that $m_1n_1+\cdots + m_{s-1}n_{s-1}+m_s(n_s-1)$ is even and that $n^\ast_i$ is odd for some $i<s$.
Let $t<s$ be the largest number such that $n^\ast_t$ is odd.
Then we have
\[
\zeta_s=
\zeta_t
\zeta_{t+1}^{m_{t+1}+m_{t+2}} 
\zeta_{t+2}^{m_{t+2} + m_{t+3}}
\cdots
\zeta_{s-1}^{m_{s-1}+m_s}.
\]
In particular, $\zeta_s=\zeta_t$ if $t+1=s$.
\end{itemize}
\end{itemize}
We denote the set of EGK data of length $n$ by $\mathcal{EGK}_n$.
Thus $\mathcal{EGK}_n\subset \coprod_{r=1}^n (\ZZ_{>0}^r \times \ZZ_{\ge 0}^r  \times {\mathcal Z}_3^r)$.
\end{definition}

Let $H=(a_1,\ldots,a_n;\vep_1,\ldots,\vep_n)$ be a naive EGK datum.
We define $n_1, n_2, \dots, n_r$ by
\begin{align*}
&a_1=\cdots=a_{n_1} <a_{n_1+1}, \\
&a_{n_1}<a_{n_1+1}=\cdots=a_{n_1+n_2}<a_{n_1+n_2+1}, \\
&\cdots \\
&a_{n_1+\cdots+n_{r-1}}<a_{n_1+\cdots+n_{r-1}+1}=\cdots =a_{n_1+\cdots+n_r}
\end{align*}
with $n=n_1+\cdots+n_r$.
For $s=1, 2, \ldots, r$, we set
\[ 
n^\ast_s=\sum_{u=1}^{s} n_u, \; m_s=a_{n^\ast_s}, \text{ and }
\zeta_s=\vep_{n^\ast_s}.
\]
The following proposition can be easily verified.
\begin{proposition} 
\label{prop.4.1}
Let $H=(a_1,\ldots,a_n;\vep_1,\ldots,\vep_n)$ be a naive $\EGK$ datum.
Then $$G=(n_1, \ldots, n_r; m_1, \ldots, m_r; \zeta_1, \ldots, \zeta_r)$$ is an $\EGK$ datum.
\end{proposition}
We define a map $\Ups=\Ups_n:\mathcal{NEGK}_n\rightarrow \mathcal{EGK}_n$ by $\Ups(H)=G$.
We call $G=\Ups(H)$ the EGK datum associated to a naive EGK datum $H$.
We also write $\Ups(\ua)=(n_1, \ldots, n_r; m_1, \ldots, m_r)$, if there is no fear of confusion.

\begin{proposition} 
\label{prop.4.2}
The map $\Ups:\mathcal{NEGK}_n\rightarrow\mathcal{EGK}_n$ is surjective.
Thus for any $\EGK$ datum 
\[
G=(n_1, \ldots, n_r; m_1, \ldots, m_r; \zeta_1, \ldots, \zeta_r)
\]
of length $n$, there exists a naive $\EGK$ datum $H$ such that $\Ups(H)=G$.
\end{proposition}

\begin{proposition} 
\label{prop.4.3} Let $F$ be a non-archimedean local field of characteristic $0$, and $\frko$ the ring of integers in $F$. 
Let  $B \in \calhnd_n(\frko )$. 
Then $\EGK(B)$ is an $\EGK$ datum of length $n$.
\end{proposition}

We  say that  $H_B$ is a naive  EGK data of  $B$ if $\Ups(H_B)=\EGK(B)$.
We note that $H_B$ is not necessarily uniquely determined by $B$.

We give an explicit formula for a naive EGK datum. We assume that $F$ is a finite unramified extension of $\QQ_2$ in Theorems \ref{th.4.1},\ref{th.4.2},\ref{th.4.3} and Example \ref{exp.4.1}.

\begin{theorem}
\label{th.4.1} Let $B=2^{k_1}C_1 \bot \cdots \bot 2^{k_r}C_r \in \calh_n(\frko)$ be  a pre-optimal form. 
Then there is an optimal basis $\{\phi_1,\ldots,\phi_n \}$ of $L_B$ such that  $\{\phi_1,\ldots,\phi_{n-n_r} \}$ is an optimal basis of $L_{B^{[r-1]}}$.
\end{theorem}
\begin{proof}
The assertion follows from Theorems \ref{th.2.1} and \ref{th.3.1}, and Proposition \ref{prop.3.1}
\end{proof}
By rewriting the above theorem we obtain, 
\begin{corollary}
\label{cor.4.1}
 Let the notation be as above.  Then there is an optimal form $\widetilde B$ which is equivalent to $B$ such that
$\widetilde B^{[i]}$ is an optimal form which is equivalent to $B^{[i]}$ for any $1 \le i \le r$.
\end{corollary}
\begin{theorem}
\label{th.4.2} Let $B=2^{k_1}C_1 \bot \cdots \bot 2^{k_r}C_r$ be a pre-optimal form of degree $n$, and $\GK(B)$ as
\[\mathrm{GK}(B)=(\underbrace{m_1,\ldots,m_1}_{n_1},\ldots,\underbrace{m_s,\ldots,m_s}_{n_s})\]
with $m_1<\cdots<m_s$ and $n=n_1+\cdots+n_{s-1}+n_s$.
For $j=1,2,\ldots,s$ put
$$n_j^\ast=\sum_{u=1}^j n_u.$$ 
Then we have the following
\begin{itemize}
\item [(1)] For any $1 \le j \le s$ there is a positive  integer $l_j \le r$ such that $B^{(n_j^*)}=B^{[l_j]}$.
\item [(2)] For each $1 \le j \le s$ define $\zeta_j$ as
$$\zeta_j=\begin{cases} \eta_{B^{[l_j]}} & \text{ if } n_j^* \text{ is odd} \\
\xi_{B^{[l_j]}} & \text{ if } n_j^* \text { is even} 
\end{cases}.$$
Then $\mathrm{EGK}(B)=(n_1,\ldots,n_s;m_1,\ldots,m_s;\zeta_1,\ldots,\zeta_s)$. 
\end{itemize}
\end{theorem}
\begin{proof}
The assertion (1) follows from Theorem \ref{th.3.1}.  Let $\widetilde B$ be that in Corollary \ref{cor.4.1}. Then, we have $B^{[l_j]} \sim \widetilde B^{[l_j]}$. Thus the assertion (2) holds.
\end{proof}

\begin{theorem}
\label{th.4.3}
Let $B=2^{k_1}C_1 \bot \cdots \bot 2^{k_r}C_r$ be a pre-optimal form of degree $n$ with $n_i = \deg C_i$. Let $\GK(B)=(a_1,\ldots,a_n)$.
We define $\vep_i$ as
\[
\vep_i=\begin{cases}
1 & \text{ if } i=1 \\
\xi_{B^{(i)}}  & \text{ if } i \text{ is even and } i=n_1+\cdots+n_s \\ { } & \text{ with some } 1 \le s \le r \\
\eta_{B^{(i)}} & \text{ if } i \text{ is odd and } i=n_1+\cdots+n_s \ge 3 \\
{ } & \text{ with some } 1 \le s \le r \\
\eta_{B^{(i+1)}}\xi_{B^{(i+1)}}^{a_i} & \text{ if } i  \text{ is odd and }  i=n_1+\cdots+n_s-1 \ge 3 \\
{ } & \text{ with some } 1 \le s \le r  \text{ such that }  n_s=2 \\
{ } & \text { and } a_1+\cdots+a_{i+1} \text { is even}\\
0 & \text{ if } i \text{ is even and }  i=n_1+\cdots+n_s-1 \\
{ } & \text{ with some } 1 \le s \le r  \text{ such that } n_s=2 \\
{ } & \text { and } a_1+\cdots+a_{i} \text{ is odd}\\
\pm 1 & \text{ if } i=n_1+\cdots+n_s-1 \ge 2  \\
{ } &  \text{ with some } 1 \le s \le r \text{ such that } n_s=2 \\
{ } &  \text{ and } i+1+a_1+\cdots+a_{2[(i+1)/2]} \text { is } odd,
\end{cases}    
\]
and, for any $1 \le s \le r$, and put $H^{[s]}=(a_1,\ldots,a_{n_1+\cdots+n_s};\vep_1,\ldots,\vep_{n_1+\cdots+n_s})$. 
Then $H^{[s]}$ is a naive $\EGK$ datum of  $B^{[s]}$.
\end{theorem}
\begin{proof}
We prove the assertion by induction on $s$. The assertion clearly holds if $s=1$. Suppose that $s >1$ and that the assertion holds for $s-1$. Take an optimal form $\widetilde B$ satisfying the condition in Corollary \ref{cor.4.1}. Put $\widetilde n_s=n_1+\cdots+n_s$.
Suppose that $\deg C_s=1$ and
$\widetilde n_s$ is even. Then by definition,  $a_1+\cdots+a_{\widetilde n_s}$ is even if and only if $\xi_{\widetilde B^{[s]}} \not=0$. This is equivalent to saying that
$\vep_{\widetilde n_s} \not=0$. This implies that $H^{[s]}$ is a naive datum. Suppose that $\deg C_s=1$ and
$\widetilde n_s$ is odd. Then $\vep_{\widetilde n_s}=\eta_{B^{[s]}}$ and we easily see that $H^{[s]}$ is a naive EGK datum if $a_1+\cdots+a_{{\widetilde n_s}-1}$ is odd.
Suppose that $a_1+\cdots+a_{{\widetilde n_s}-1}$ is even.
Then $\vep_{\widetilde n_s-2}=\eta_{B^{[s-2]}}$ or $\eta_{B^{[s-1]}}\xi_{B^{[s-1]}}^{a_{\widetilde n_s-1}}$ according as 
$\deg C_{s-1}=1$ or $2$. In any case, by \cite{IK1}, Lemma 3.4, we have
$\vep_{\widetilde n_s}=\vep_{\widetilde n_s-2}\vep_{\widetilde n_s-1}^{ a_{\widetilde n_s}+a_{\widetilde n_s-1}}$. This implies that $H^{[s]}$ is a naive EGK datum. Similarly, we can prove that $H^{[s]}$ is a naive EGK datum in the case $\deg C_s=2$. Moreover in any case, we can easily prove that $\Upsilon_{\widetilde n_s}(H^{[s]})=\EGK(B^{[s]})$.
This completes the induction.
\end{proof}

\begin{example}
\label{exp.4.1} 
From now on for $i=1,2,3,4$ let $u_i \in \frko^{\times}$ and $K_i \in \frac{1}{2}(GL_2(\frko) \cap S_2(\frko)_e)$.\\
\begin{itemize}
\item[(1)] Let $B=2^{k_1}K_1 \bot 2^{k_2}K_2$ with $k_1 \le k_2$. Then, 
\[H_B=(k_1,k_1,k_2,k_2;1,\xi_{B^{(2)}},\eta_B\xi_B^{k_2},\xi_B)\] is a naive EGK datum of $B$.
\item[(2)] Let $B=2^{k_1}K_1 \bot 2^{k_2}u_2 \bot 2^{k_3}u_3$ with $k_1 \le k_2 \le k_3$, and put 
$H_B=(k_1,k_1,k_2,k_3';1,\xi_{B^{(2)}},\vep_3,\xi_B)$, where 
$$(k_3';\vep_3)=\begin{cases}
(k_3;\pm1) & \text{ if } k_2+k_3 \not\equiv 0 \text{ mod } 2 \\
(k_3+1;\pm1) & \text{ if } k_2+k_3 \equiv 0 \text{ mod } 2 \text{ and } u_2u_3 \equiv 1 \text{ mod } 4\\
(k_3+2;\eta_B\xi_B^{k_3}) & \text{ if } k_2+k_3 \equiv 0 \text{ mod } 2 \text{ and } u_2u_3 \equiv -1 \text{ mod } 4.
\end{cases}$$
Then, $H_B$ is a naive EGK datum of $B$. 
\item[(3)] Let $B=2^{k_1}u_1 \bot 2^{k_2}K_2 \bot 2^{k_3}u_3$ with $k_1 \le k_2 \le k_3+1$ and suppose that $B$ is pre-optimal. Put 
$H_B=(k_1,k_2,k_2,k_3';1,\vep_2,\eta_{B^{(3)}},\xi_B),$ where 
$$k_3'=\begin{cases}
k_3 & \text{ if } k_1+k_3 \not\equiv 0 \text{ mod } 2 \\
k_3+1& \text{ if } k_2+k_3 \equiv 0 \text{ mod } 2 \text{ and } u_2u_3 \equiv 1 \text{ mod } 4\\
k_3+2 & \text{ if } k_2+k_3 \equiv 0 \text{ mod } 2 \text{ and } u_2u_3 \equiv -1 \text{ mod } 4,
\end{cases}$$
and
$$\vep_2=\begin{cases}
0 & \text{ if } k_1+k_2 \not\equiv 0 \text{ mod } 2 \\
\pm1 & \text{ if } k_1+k_2 \equiv 0 \text{ mod } 2.
\end{cases}$$
Then $H_B$ is a naive EGK datum of $B$. 
\item[(4)] Let  $B=2^{k_1}u_1 \bot 2^{k_2}u_2 \bot 2^{k_3}K_3$ with $k_1 \le k_2 < k_3$ and suppose that $B$ is pre-optimal. Put
$H_B=(k_1,k_2',k_3,k_3;1,\xi_{B^{(2)}},\vep_3,\xi_B)$, where 
$$(k_2';\vep_3)=\begin{cases}
(k_2;\pm 1) & \text{ if } k_1+k_2 \not\equiv 0 \text{ mod } 2 \\
(k_2+1;\pm1) & \text{ if } k_1+k_2 \equiv 0 \text{ mod } 2 \text{ and } u_1u_2 \equiv 1 \text{ mod } 4\\
(k_2+2;\eta_B\xi_B^{k_3}) & \text{ if } k_1+k_2 \equiv 0 \text{ mod } 2 \text{ and } u_1u_2 \equiv -1 \text{ mod } 4.
\end{cases}$$
Then $H_B$ is a naive EGK datum of $B$. 
\item[(5)] $B=2^{k_1}u_1 \bot 2^{k_2}u_2 \bot 2^{k_3}u_3 \bot 2^{k_3}u_4$ with $k_1 \le k_2< k_3 \le k_4$ and 
suppose that $B$ is pre-optimal. 
\begin{itemize}
\item [(5.1)] Suppose that $k_3=k_4$ and $\xi_{B^{(2)}}=\xi_B=0$.  Put
$H_B=(k_1,k_2',k_3+1,k_3+1;1,0,\pm 1,0)$, where 
$$k_2'=\begin{cases}
k_2    & \text{ if } k_1+k_2 \not\equiv 0 \text{ mod } 2 \\
k_2+1& \text{ if } k_1+k_2 \equiv 0 \text{ mod } 2.
\end{cases}$$
Then $H_B$ is a naive EGK datum of $B$.
\item[(5.2)] Suppose that $B$ does not satisfy the condition (5.1). 
Put  $H_B=(k_1,k_2',k_3',k_4';1,\xi_{B^{(2)}},\eta_{B^{(3)}},\xi_B)$, where
$$(k_2',k_3',k_4')=\begin{cases}
(k_2,k_3+2,k_4+2)    & \text{ if } k_1+k_2 \not\equiv 0 \text{ mod } 2 \\
(k_2+1,k_3+1,k_4+2) & \text{ if } k_1+k_2 \equiv 0 \text{ mod } 2 \text{ and } \xi_{B^{(2)}}=0\\
(k_2+2,k_3,k_4+1) & \text{ if }  \xi_{B^{(2)}}\not=0 \text{ and } \xi_B=0\\
(k_2+1,k_3,k_3+2) & \text{ if }  \xi_{B^{(2)}}\not=0 \text{ and } \xi_B\not=0.
\end{cases}$$
Then $H_B$ is a naive EGK datum of $B$.
\end{itemize}
\item[(6)] Let $B=2^{k_1}u_1 \bot 2^{k_2}u_2 \bot 2^{k_2}u_3 \bot 2^{k_3}u_4$ with $k_1 < k_2 < k_3$ and suppose that $B$ is pre-optimal. 
\begin{itemize}
\item[(6.1)] Suppose that $k_1+k_2$ is even.  
Put  $H_B=(k_1,k_2+1,k_2+1,k_3';1,0,\eta_{B^{(3)}},\xi_B)$, where
$$k_3'=\begin{cases}
k_3& \text{ if } k_1+k_3 \not\equiv 0 \text{ mod } 2\\
k_3+1 & \text{ if }  k_1+k_3 \equiv 0 \text{ mod } 2 \text { and } \xi_B=0\\
k_3+2 & \text{ if }   \xi_B\not=0.
\end{cases}$$
Then $H_B$ is a naive EGK datum of $B$.
\item[(6.2)] Suppose that  $k_1+k_2$ is odd.  
Put  $H_B=(k_1,k_2,k_2+2,k_3';1,0,\eta_{B^{(3)}},\xi_B)$, where
$$k_3'=\begin{cases}
k_3& \text{ if } k_1+k_3 \not\equiv 0 \text{ mod } 2\\
k_3+1 & \text{ if }  k_1+k_3 \equiv 0 \text{ mod } 2 \text { and } \xi_B=0\\
k_3+2 & \text{ if }   \xi_B\not=0.
\end{cases}$$
Then $H_B$ is a naive EGK datum of $B$.
\end{itemize}
\end{itemize}
\end{example}

Finally, we recall a result in the non-dyadic case (cf. \cite{IK1}, Proposition 6.1).

\begin{theorem}
\label{th.4.4}
Assume  that $F$ is non-dyadic field.
Let $T=(t_1)\perp \cdots \perp (t_n)$ be a diagonal matrix such that $\ord(t_1)\leq \ord(t_2)\leq \cdots \leq \ord(t_n)$.
Put $a_i=\ord(t_i)$ and 
\[
\vep_i=\begin{cases} 
\xi_{T^{(i)}} & \text{ if $i$ is even,} \\
\eta_{T^{(i)}} & \text{ if $i$ is odd.} 
\end{cases}
\]
Then $(a_1, \ldots, a_n; \vep_1, \ldots, \vep_n)$ is a naive EGK datum of $T$.
\end{theorem}

\section{Explicit formula for the Siegel series}
 In this section we give an explicit formula for the Siegel series in terms of naive EGK data in the case $F=\QQ_p$.
Once for all, we fix an additive character ${\bf e}_p:\QQ_p \longrightarrow \CC^\times$ of order zero, that is, a character such that
\[\frko =\{ a \in \QQ_p \  | \ {\bf e}_p(ax)=1 \ \text{ for  any} \ x \in \ZZ_p \}.\]
For  a half-integral matrix $B$ of degree $n$ over $\ZZ_p$ define the local Siegel series $b(B,s)$ by 
\[b(B,s)= \sum_{R} {\bf e}_p({\rm tr}(BR))\mu(R)^{-s},\]
where $R$ runs over a complete set of representatives of ${\rm Sym}_n(\QQ_p)/{\rm Sym}_n(\ZZ_p)$ and $\mu(R)=[R\ZZ_p^n+\ZZ_p^n:\ZZ_p^n]$. For a non-degenerate half-integral matrix $B$ of degree $n$ over $\ZZ_p$ define a polynomial $\gamma(B,X)$ in $X$ by 
\[\gamma(B,X)=
\begin{cases}
(1-X)\prod_{i=1}^{n/2}(1-p^{2i}X^2)(1-p^{n/2}\xi_B X)^{-1} & \text{ if $n$ is  even} \\ 
(1-X)\prod_{i=1}^{(n-1)/2}(1-p^{2i}X^2) & \text{ if $n$ is  odd.} \end{cases}\] 
Then it is shown by \cite{shimura97} that there exists a polynomial $F(B,X)$ in $X$ with coefficients in $\ZZ$ such that 
\[F(B,p^{-s})=\frac{b(B,s)}{\gamma(B,p^{-s})}.\]
 We define a symbol $X^{1/2}$ so that $(X^{1/2})^2=X$.
We define $\widetilde F(B,X)$ as
\[\widetilde F(B,X)=X^{-\frke_B/2}F(B,p^{-(n+1)/2}X).\]
We note that $\widetilde F(B,X) \in \QQ[p^{1/2}][X,X^{-1}]$ if $n$ is even, and
$\widetilde F(B,X) \in \QQ[X^{1/2},X^{-1/2}]$ if $n$ is odd.

\begin{definition} 
\label{def.5.1} 
For integers $e,\widetilde e$,  a real number $\xi$, we  define  rational functions $C(e,\widetilde e,\xi;X)$ 
and $D(e,\widetilde e,\xi;X)$ in  $X^{1/2}$ by 
\[C(e,\widetilde e,\xi;X)={p^{\widetilde e/4}X^{-(e- \widetilde e)/2-1}(1-\xi p^{-1/2} X)  \over X^{-1}-X} \]
and
\[D(e,\widetilde e,\xi;X)= {p^{\widetilde e/4}X^{-(e-\widetilde e)/2}   \over 1- \xi X} .\]
\end{definition}
For a positive integer $i$  put 
$$C_i(e,\widetilde e,\xi;X)= \begin{cases}
C(e,\widetilde e,\xi;X) &  \text { if  $i$ is even } \\
D(e,\widetilde e,\xi;X) & \text{ if $i$ is odd.}
\end{cases}.$$
\begin{definition}
\label{def.5.2}
For a sequence $\underline a=(a_1,\ldots,a_n)$ of integers and an integer $1 \le i \le n$, we define $\frke_i=\frke_i(\underline a)$ as
$$\frke_i=
\begin{cases} a_1+\cdots +a_i  & \text{ if  $i$ is odd} \\
2[(a_1+\cdots+a_i)/2] & \text{ if $i$ is even.}
\end{cases}$$
 We also put $\frke_0=0$.

For a naive EGK datum $H=(a_1,\ldots,a_n;\vep_1,\ldots,\vep_n)$ we define a rational function $\calf(H;X)$ in $X^{1/2}$ as follows:
First we define
$$\calf(H;X)=X^{-a_1/2}+X^{-a_1/2+1}+\cdots+X^{a_1/2-1}+X^{a_1/2}$$
if $n=1$. Let  $n>1$. Then $H'= (a_1,\ldots,a_{n-1};\vep_1,\ldots,\vep_{n-1})$ is a naive EGK datum of length $n-1$. 
Suppose that $\calf(H';X)$ is defined for $H'$. Then, we define $\calf(H;X)$ as
\begin{align*}
&\calf(H;X)=C_n(\frke_n,\frke_{n-1},\xi;X)\calf(H';p^{1/2}X)\\
&+\zeta C_n(\frke_n,\frke_{n-1},\xi;X^{-1})\calf(H';p^{1/2}X^{-1}),
\end{align*}
where $\xi=\vep_n$ or $\vep_{n-1}$ according as $n$ is even or odd, and $\zeta=1$ or $\vep_n$ according as $n$ is even or odd.

\end{definition}

First we easily see  the following (e.g. \cite[Example (1)]{Kat}).
\begin{proposition}
\label{prop.5.1}
Let  $B =(b) \in \calh_1(\ZZ_2)$ with $\ord(b)=a_1$. 
Then, we have
\[
\widetilde F(B,X)=\sum_{i=0}^{a_1} X^{i-(a_1/2)}.
\]
\end{proposition}
We give induction formulas for $\widetilde F(B,X)$.   First we review induction formulas in \cite{Kat}.

\begin{theorem}
\label{th.5.1}
Let $B=2^{k_1}J_1 \bot \cdots 2^{k_r}J_r \in \calh_n(\ZZ_2)$ be a weak canonical form with $J_r$ a unimodular diagonal matrix.

\begin{itemize}
\item[(1)] Suppose that $\deg J_r=1$. Put
\[a_n=\begin{cases} 
k_r+2 & \text{ if } n \text{ is odd and } \ord(\det B^{(n-1)}) \text{ is odd} \\
{ } & \text{ or } n \text{ is even } \xi_{B} \not=0 \\
k_r+1 & \text{ if } n \text{ is odd, } \ord(\det B^{(n-1)}) \text{ is even and } \xi_{B^{(n-1)}}=0 \\
{ } & \text{ or } n \text{ is even and } \xi_{B}=0 \\
k_r & \text{ if } n \text{ is odd and } \xi_{B^{(n-1)}}\not=0 \\
{ } & \text{ or } n \text{ is even and } \ord(\det B) \text{ is odd},
\end{cases} 
\] 
and $\frke_i=\frke_i(GK(B^{(n-1)},a_n)$ for $i=n-1,n$.  Then we have
\begin{align*}
\widetilde F(B,X)
=&C_n(\frke_n,\frke_{n-1},\xi;X)\widetilde F(B^{(n-1)}, 2^{1/2} X) \notag \\
&+\eta C_n(\frke_n,\frke_{n-1},\xi;X^{-1})\widetilde F(B^{(n-1)},2^{1/2} X),\notag
\end{align*}
where $\xi=\begin{cases} \xi_B & \text{ if } n \text{ is even}\\
\xi_{B^{(n-1)}} & \text{ if } n \text{ is odd},
\end{cases}$
and 
$\eta=\begin{cases} 1 & \text{ if } n \text{ is even}\\
\eta_{B} & \text{ if } n \text{ is odd.}
\end{cases}$
\item[(2)] Suppose that $J_r=2$ and put $J_r=u_1 \bot u_2$.
\begin{itemize}
\item[(2.1)] Suppose that either $n$ is even and $\xi_B=\xi_{B^{(n-2)}}=0$, or  $n$ is odd and $\mathrm{ord}(\det B^{(n-2)})+k_r$ is even. Let $(a_{n-1},a_n)=(k_r,k_r)$ and
$\frke_i =\frke_i(GK(B^{(n-2)},a_{n-1},a_n)$ for $i=n-2,n-1,n$. Then we have 
\begin{align*}
& \widetilde F(B,X) \\
&=C_n(\frke_n,\frke_{n-1},0;X)C_{n-1}(\frke_{n-1},\frke_{n-2},0;2^{1/2}X) \widetilde F(B^{(n-2)},2X) \\
&+C_n(\frke_n,\frke_{n-1},0;X^{-1})C_{n-1}(\frke_{n-1},\frke_{n-2},0;2^{1/2}X^{-1}) \widetilde F(B^{(n-2)},2X^{-1}) \\
&+\{ C_n(\frke_n,\frke_{n-1},0;X)\eta' C_{n-1}(\frke_{n-1},\frke_{n-2},0;(2^{1/2}X)^{-1})\\
&+\eta C_n(\frke_n,\frke_{n-1},0;X^{-1}) C_{n-1}(\frke_{n-1},\frke_{n-2},0;(2^{1/2}X^{-1})^{-1})\}
\widetilde F(B^{(n-2)},X)
\end{align*}
where 
$\eta=\begin{cases} 1 & \text{ if } n \text{ is even}\\
\eta_{B} & \text{ if } n \text { is odd,}
\end{cases}$\\
and 
$$\eta'=\begin{cases} \pm 1 &  \text{ if } n \text{ is even  } \\
\eta_{B^{(n-2)}} & \text{ if } n \text{ is odd.}
\end{cases}$$
\item[(2.2)] Suppose that $B$ does not satisfy the condition in (2.1). If $n$ is even, put 
$$(a_{n-1},a_n)=
\begin{cases}
(k_r+1,k_r+2) & \text{ if }  \xi_{B^{(n-2)}}=0 \text{ and } \xi_B \not=0 \\
(k_r,k_r)  & \text{ if } \xi_{B^{(n-2)}}\not=0 \text{ and } \mathrm{ord}(\det B) \text{ is odd} \\
(k_r,k_r+1)  & \text{ if } \xi_{B^{(n-2)}}\not=0 \text{ and } \mathrm{ord}(\det B) \text{ is even } \\
{} & \text{  and } \xi_B=0\\
(k_r,k_r+2)  & \text{ if } \xi_{B^{(n-2)}}\not=0 \text{ and }  \xi_B\not=0.
\end{cases}$$
If $n$ is odd, put $(a_{n-1},a_n)=(k_r,k_r+2)$.
In both cases, put 
$\frke_i=\frke_i((\GK(B^{(n-2)}),a_{n-1},a_n)$ for $i=n-2,n-1,n$.
  Then we have 
\begin{align*}
& \widetilde F(B,X) \\
&=C_n(\frke_n,\frke_{n-1},\xi_B;X)C_{n-1}(\frke_{n-1},\frke_{n-2},\xi_{B^{(n-2)}};2^{1/2}X) \widetilde F(B^{(n-2)},2X) \\
&+C_n(\frke_n,\frke_{n-1},\xi_Bi;X^{-1})C_{n-1}(\frke_{n-1},\frke_{n-2},\xi_{B^{(n-2)}};2^{1/2}X^{-1}) \widetilde F(B^{(n-2)},2X^{-1}) \\
&+\{ C_n(\frke_n,\frke_{n-1},\xi_B;X)\eta_{B^{(n-1)}}C_{n-1}(\frke_{n-1},\frke_{n-2},0;(2^{1/2}X)^{-1})\\
&+ C_n(\frke_n,\frke_{n-1},\xi_B;X^{-1}) C_{n-1}(\frke_{n-1},\frke_{n-2},0;(2^{1/2}X^{-1})^{-1})\}
\widetilde F(B^{(n-2)},X).
\end{align*}

\end{itemize}
\end{itemize}
\end{theorem}

\begin{proof}
(1) follows from    \cite[Theorem 4.1]{Kat}.
(2) follows from   \cite[Theorem 4.2]{Kat}.
\end{proof}

\begin{theorem}
\label{th.5.2}
Let $B=2^{k_1}J_1 \bot \cdots 2^{k_r}J_r \in \calh_n(\ZZ_2)$ be a weak canonical form, and suppose that 
$k_{s+1}=\cdots=k_r$ and $J_{s+1}, \cdots, J_r \in \frac{1}{2}(S_2(\ZZ_2)_e \cap GL_2(\ZZ_2))$ for some $0 \le s \le r-1$.

\begin{itemize}
\item[(1)] Let $k_s=k_r-1$, $J_s=\bot_{j=1}^{n_s} u_j$  with $1 \le n_s \le 2$ and $u_1,u_2 \in \frko^{\times}$ and  
put 
\[B_1=2^{k_1} J_1 \bot \cdots \bot 2^{k_{s-1}}J_{s-1} \bot 2^{k_s} u_{n_s-1} \bot 2^{k_r} J_{s+1} \bot \cdots \bot 2^{k_r} J_r.\]
 Here we make the convention that 
\[2^{k_1} J_1 \bot \cdots \bot 2^{k_{s-1}}J_{s-1} \bot 2^{k_s} u_{n_s-1}=2^{k_1} J_1 \bot \cdots \bot 2^{k_{s-1}}J_{s-1}\] if $n_s=1$. Suppose that one of the following conditions hold:
\begin{itemize}
\item[(1.1)] $n$ is even and $\mathrm{ord}(\det B)$ is even.
\item[(1.2)] $n$ is odd, $\xi_{B_1}=0$.
\end{itemize} 
Put
$$a_n=\begin{cases} k_r+1 & \text{ if } n \text{ is even and } \xi_B \not=0\\
k_r  & \text{ if } n \text{ is even and } \xi_B=0\\
k_r+1 & \text{ if } n \text{ is odd and } \mathrm{ord}(\det(B_1)) \text{ is odd} \\
k_r & \text{ if } n \text{ is odd and } \mathrm{ord}(\det(B_1)) \text{ is even,} 
\end{cases}$$
and
$\frke_i=\frke_i((\GK(B_1),a_n))$ for $i=n-1,n$. 
Then we have
\begin{align*}
\widetilde F(B,X)
=&C_n(\frke_n,\frke_{n-1},\xi;X)\widetilde F(B_1, 2^{1/2} X) \notag \\
&+\eta C_n(\frke_n,\frke_{n-1},\xi;X^{-1})\widetilde F(B_1,2^{1/2} X),\notag
\end{align*}
where $\xi=\begin{cases} \xi_B & \text{ if } n \text{ is even}\\
\xi_{B_1} & \text{ if } n \text{ is odd},
\end{cases}$

and 

$\eta=\begin{cases} 1 & \text{ if } n \text{ is even}\\
\eta_{B} & \text{ if } n \text{ is odd}.
\end{cases}$
\item[(2)] Suppose that $B$ does not satisfy the condition in (1). Let $(a_{n-1},a_n)=(k_r,k_r)$ and $\frke_i=\frke_i(GK(B^{(n-2)},a_{n-1},a_n)$. 
Then 
\begin{align*}
& \widetilde F(B,X) \\
&=C_n(\frke_n,\frke_{n-1},\xi;X)C_{n-1}(\frke_{n-1},\frke_{n-2},\xi';2^{1/2}X) \widetilde F(B^{(n-2)},2X) \\
&+C_n(\frke_n,\frke_{n-1},\xi;X^{-1})C_{n-1}(\frke_{n-1},\frke_{n-2},\xi';2^{1/2}X^{-1}) \widetilde F(B^{(n-2)},2X^{-1}) \\
&+\{ C_n(\frke_n,\frke_{n-1},\xi;X)\eta' C_{n-1}(\frke_{n-1},\frke_{n-2},\xi';(2^{1/2}X)^{-1})\\
&+\eta C_n(\frke_n,\frke_{n-1},\xi;X^{-1}) C_{n-1}(\frke_{n-1},\frke_{n-2},\xi';(2^{1/2}X^{-1})^{-1})\}
\widetilde F(B^{(n-2)},X)
\end{align*}
where 
$$\xi=\begin{cases} \xi_{B} & \text{ if } n \text{ is even}\\
0  & \text{ if } n \text{ is odd and } a_1+\cdots+a_{n-1} \text{ is odd} \\
\pm 1 & \text{ if } n \text { is odd and } a_1+\cdots+a_{n-1} \text{ is odd,} 
\end{cases}$$
$$\eta=\begin{cases} 1 & \text{ if } n \text{ is even}\\
\eta_{B} & \text{ if } n \text { is odd,}
\end{cases}$$
and 
$$\xi'=\begin{cases} \xi_{B^{(n-2)}} & \text{ if } n \text{ is even}\\
0 & \text{ if } n \text{ is odd and } a_1+\cdots+a_{n-1} \text{ is odd}  \\
\pm 1 &  \text{ if } n \text { is odd and  }  a_1+\cdots+a_{n-1} \text{ is even,} 
\end{cases}$$
$$\eta'=\begin{cases} \eta_{B}\xi_{B}^{a_n} & \text{ if } n \text{ is even and } a_1+\cdots+a_n \text{ is even }\\
\pm 1 &  \text{ if } n \text{ is even and  }  a_1+\cdots+a_n \text{ is odd} \\
\eta_{B^{(n-2)}} & \text{ if } n \text{ is odd.}
\end{cases}$$

\end{itemize}
\end{theorem}

\begin{proof}

%\item[(3)] 
The assertion (1) for the case $\deg J_r=1$
follows from \cite[Theorem 4.1]{Kat}, and 
the assertion (1) for the case $\deg J_r=2$
can also be proved by using the same argument as in the proof of \cite[Theorem 4.2]{Kat}.
 (2) follows from  \cite[Theorem 4.2]{Kat}.

\end{proof}

By rewriting Theorems \ref{th.5.1} and \ref{th.5.2}, we obtain the following two theorems.

\begin{theorem}
\label{th.5.3}
Let $B=2^{k_1}J_1 \bot \cdots \bot 2^{k_r}J_r \in \calh_n(\ZZ_2)$ be a pre-optimal form.  Suppose that $\deg J_r=1$.  Let $a_n$ be that defined in Theorem \ref{th.3.1} (2), and put  $\frke_i=\frke_i((GK(B^{(n-1)}),a_n))$ for $i=n-1,n$. Then
\begin{align}
\widetilde F(B,X)
=&C_n(\frke_n,\frke_{n-1},\xi;X)\widetilde F(B^{(n-1)}, 2^{1/2} X) \notag \\
&+\eta C_{n-1}(\frke_n,\frke_{n-1},\xi;X^{-1})\widetilde F(B^{(n-1)},2^{1/2} X),\notag
\end{align}
where    
$\xi=\begin{cases} \xi_B & \text{ if } n \text{ is even} \\
\xi_{B^{(n-1)}} & \text{ if } n \text{ is odd},
\end{cases}$ 

 and 

$\xi=\begin{cases} 1 & \text{ if } n \text{ is even} \\
\eta_{B} & \text{ if } n \text{ is odd}.
\end{cases}$
\end{theorem}

\begin{theorem} 
\label{th.5.4}
Let $B=2^{k_1}J_1 \bot \cdots \bot 2^{k_r}J_r$ be a pre-optimal form.  Suppose that $\deg J_r=2$.  Let $(a_{n-1},a_n)$ be that defined in Theorem \ref{th.3.1} (1),(3), and put  $\frke_i=\frke_i((GK(B^{(n-2)}),a_{n-1},a_n))$ for $i=n-2,n-1,n$.   Then,
\begin{align*}
& \widetilde F(B,X) \\
&=C_n(\frke_n,\frke_{n-1},\xi;X)C_{n-1}(\frke_{n-1},\frke_{n-2},\xi';2^{1/2}X) \widetilde F(B^{(n-2)},2X) \\
&+C_n(\frke_n,\frke_{n-1},\xi;X^{-1})C_{n-1}(\frke_{n-1},\frke_{n-2},\xi';2^{1/2}X^{-1}) \widetilde F(B^{(n-2)},2X^{-1}) \\
&+\{ C_n(\frke_n,\frke_{n-1},\xi;X)\eta' C_{n-1}(\frke_{n-1},\frke_{n-2},\xi';(2^{1/2}X)^{-1})\\
&+\eta C_n(\frke_n,\frke_{n-1},\xi;X^{-1}) C_{n-1}(\frke_{n-1},\frke_{n-2},\xi';(2^{1/2}X^{-1})^{-1})\}
\widetilde F(B^{(n-2)},X)
\end{align*}
where 
$$\xi=\begin{cases} \xi_{B} & \text{ if } n \text{ is even}\\
0  & \text{ if } n \text{ is odd and } a_1+\cdots+a_{n-1} \text{ is odd} \\
\pm 1 & \text{ if } n \text { is odd and } a_1+\cdots+a_{n-1} \text{ is odd,} 
\end{cases}$$
$$\eta=\begin{cases} 1 & \text{ if } n \text{ is even}\\
\eta_{B} & \text{ if } n \text { is odd,}
\end{cases}$$
and 
$$\xi'=\begin{cases} \xi_{B^{(n-2)}} & \text{ if } n \text{ is even}\\
0 & \text{ if } n \text{ is odd and } a_1+\cdots+a_{n-1} \text{ is odd}  \\
\pm 1 &  \text{ if } n \text { is odd and  }  a_1+\cdots+a_{n-1} \text{ is even,} 
\end{cases}$$ 
$$\eta'=\begin{cases} \eta_{B}\xi_{B}^{a_n} & \text{ if } n \text{ is even and } a_1+\cdots+a_n \text{ is even }\\
\pm 1 &  \text{ if } n \text{ is even and  }  a_1+\cdots+a_n \text{ is odd} \\
\eta_{B^{(n-2)}} & \text{ if } n \text{ is odd.}
\end{cases}$$
\end{theorem}

 Now we give an explicit formula for the Siegel series of $B$ 
in terms of its naive EGK datum.

\begin{theorem}
\label{th.5.5} Let $B \in \calh_n(\ZZ_p)$. Then there exists a naive EGK datum $H$ such that
\[\widetilde F(B,X)=\calf(H;X). \]
\end{theorem}
\begin{proof} First assume that $p=2$. We prove the assertion by induction on $n$. 
The assertion holds  for $B \in \calh_1(\ZZ_2)$ by Proposition \ref{prop.5.1}. Let $n \ge 2$ and assume that the assertion holds for any $n'<n$ and $B' \in \calh_{n'} (\ZZ_2)$. Let $B \in \calh_n(\ZZ_2)$. We may assume that 
$B=2^{k_1} C_1 \bot \cdots \bot 2^{k_r}C_r$ is a pre-optimal form.
Take the naive EGK datum $H=H^{[r]}=(a_1,\ldots,a_n;\vep_1,\ldots,\vep_n)$ of $B$ as in Theorem \ref{th.4.3}. 
First assume that $\deg C_r=1$. Then $H^{[r-1]}$ is a naive EGK datum of $B^{(n-1)}$. We note that 
$(\vep_{n-1},\vep_n)=(\xi_B,1)$ or $(\eta_B,\xi_{B^{(n-1)}})$ according as $n$ is even or odd. Hence, by Theorem \ref{th.5.3} combined with the induction assumption, we have
\begin{equation}
\label{eq.1}
\widetilde F(B,X)=\calf(H;X).
\end{equation} 
 Similarly, the equality (\ref{eq.1})
 for $\deg C_r=2$ follows from Theorem \ref{th.5.4}. 
This completes the induction.

Next assume that $p$ is odd. Then, the assertion can be proved in the same manner as above by  \cite[Theorem 4.1]{Kat}, Theorem \ref{th.4.4} and Proposition \ref{prop.5.1}.

\end{proof}

\begin{remark}
\begin{itemize}
\item[(1)] $\calf(H,X)$ is a Laurent polynomial in $X^{1/2}$ and
it is uniquely determined by $\EGK(B)$ (cf. \cite[Theorem 4.1]{IK2}). 
Therefore,  Theorem \ref{th.5.5} implies that $\widetilde F(B,X)$ can be expressed explicitly in terms of $\EGK(B)$. This holds  for any non-archimedean local field $F$ of characteristic $0$ and any $B \in \calh_n(\frko)$ (cf. \cite[Theorem 1.1]{IK2}).
\item [(2)] Assertions similar to Theorems \ref{th.5.1} and 
\ref{th.5.2} can easily be proved by using the same argument in the proofs of \cite[Theorems 4.1 and 4.2]{Kat} in the case that
$F$ is a non-dyadic field or a finite unramified extension of $\QQ_2$.  Therefore, Theorem \ref{th.5.5} can be proved in this case without using the results of \cite{IK2}. 
\end{itemize}
\end{remark}

\section{Matrix reductions for weak canonical forms}\label{sec:reduction}
The purpose of this section is to elaborate the key mechanisms behind our computer program \textit{computeGK}.
To make use of explicit formulas in the previous sections, it is required to find a pre-optimal form which in turn could be obtained from a weak canonical form by the argument of Proposition \ref{prop.2.1}. It is straightforward to turn this argument into a computer code. Thus, in this section we discuss the detailed procedures to obtain a weak canonical decomposition for arbitrarily given $B\in \mat{n}{2}$ for the purpose of algorithmic computation. 
%These are somewhat elementary in nature, but form crucial components of \textit{computeGK}. 
\\

{\bf Jordan decomposition.}
%\label{subsec:jordan}
Let $B\in \mat{n}{p}$, where $p$ is not necessarily $2$. Let us review how to find an orthogonal sum $B'$ of Jordan components such that $B \sim B'$. By a \textit{Jordan component}, here we mean a matrix of degree 1, or a matrix of the form $(b_{ij})\in \mat{2}{2}$ such that $\ord(b_{11}),\ord(b_{22}) > \ord(b_{12})$ when $p=2$. We call an orthogonal sum of the matrices of the latter form is of type II. Although such a procedure is well-known (cf. \cite{MR0118704,MR522835}), 
we have included a simple procedure for completeness, which is adopted from the code used in \cite{MR3739221}.

Let $B=(b_{ij})_{1\leq i,j\leq n}$ be given with $n\geq 2$.

Step 0 :  By rearranging the rows and columns of $B$, we can assume that $B$ has an entry in the first row which has minimum order among all $b_{ij}$. Let $j_0$ be the smallest index among $1\leq j\leq n$ such that $b_{1j}$ has minimal order.

Step 1 :  If $j_0=1$, then skip Step 2 and proceed to Step 3. If $j_0\geq 3$, we can move $b_{1j_0}$ to the second column by reordering the indices. Because such an operation preserves the set of diagonal (and non-diagonal), it cannot be moved to the position of the first diagonal. Hence, we can assume that $j_0=2$, i.e., $b_{12}$ has minimum order among all $b_{ij}$.

Step 2 :  This step is only necessary if $p$ is odd. Let $U = E_{12}(1)$, where $E_{ij}(a)\in GL_n(\ZZ_p)$ with $a\in \ZZ_p$ denotes the matrix whose diagonal entries are 1 and the only non-zero non-diagonal entry is $a$ at the position $(i,j)$. Then $B'=(b_{ij}')$ defined by $B' = U B U^{t}$ has entries $b_{11}' = b_{11}+2b_{12}+b_{22}$, $b_{1j}' = b_{1j}+b_{2j},\, b_{j1}'=b_{1j}'\, (j>1)$, and $b_{ij}'=b_{ij},\, (i,j>1)$. Then $\ord(b_{11}')=\ord(b_{12})$ and $\ord(b_{1j}')=\ord(b_{j1}')\geq \ord(b_{12}+b_{22})=\ord(b_{12})$ for $j>1$. Therefore, $b_{11}'$ has minimum order among all $b_{ij}'$. Now we redefine $B$ as $B'$, and thus, $b_{11}$ has minimum order.

Step 3 : Let $B_0 : = (b_{11})$ if $b_{11}$ has minimum order, and  $B_0: = \begin{pmatrix}
  b_{11} & b_{12} \\
  b_{21} & b_{22}
  \end{pmatrix}$
otherwise, i.e., $b_{12}$ has minimum order with $\ord(b_{11})>\ord(b_{12})$. Let $m=\deg B_0\in \{1,2\}$. So $B$ takes the form
$$
B = 
\left(\begin{array}{@{}c|c@{}}
B_0& X^{t} \\
\hline
  X & B_1
\end{array}\right)
$$
for some $(n-m)\times m$ matrix $X$ and $B_1$ of degree $n-m$. Define a matrix of degree $n$ by
$$
U = \left(\begin{array}{@{}c|c@{}}
I_{m}& 0 \\
\hline
  XB_0^{-1} & I_{n-m}
\end{array}\right).
$$
Now we claim that $U\in GL_n(\ZZ_p)$. Note that $XB_0^{-1}$ has entries in $\ZZ_p$. Indeed, when $B_0$ has degree 1 this is due to the minimality of $\ord(b_{11})$. When $\deg B_0=2$, 
$c_{12} =-b_{12}(b_{11}b_{22}-b_{12}^2)^{-1}$ has minimum order $\ord(c_{12}) = -\ord(b_{12})$ among all entries of $B_0^{-1}=(c_{ij})$. Thus, $XB_0^{-1}$ has entries in $\ZZ_p$ in any case. Since $U$ is a lower triangular matrix with unit diagonal, its inverse also has entries in $\ZZ_p$. Hence, $U\in GL_n(\ZZ_p)$.

We obtain
$$
B' = UBU^{t}=\left(\begin{array}{@{}c|c@{}}
B_0& 0 \\
\hline
0 & B_1'
\end{array}\right)
$$
with $B_1'\in \mat{n-m}{p}$.

Repeating Steps 0-3 with matrices of smaller size, we eventually obtain a Jordan decomposition of $B$.

{\bf Watson reduction.}
Now assume that $p=2$ and $B\in \mat{n}{2}$ is an orthogonal sum of the Jordan components. 
Let us explain how to find a matrix, $GL_n(\ZZ_2)$-equivalent to $B$, of the following form in an explicit manner :
\begin{equation}\label{eqn:Watson}
2^{e_1}(V_1\bot U_1) \bot \ldots \bot 2^{e_m}(V_m \bot U_m)
\end{equation}
with $0\leq e_1 <\cdots < e_m$, $U_i = \emptyset$ or $(u_1)$ or $(u_1) \bot (u_2)$ for $u_1,u_2 \in \{1,3,5,7\} $, and $V_i = \emptyset$ or $H\bot \cdots \bot H$ or $H \bot \cdots \bot H \bot Y$ (see, for instance, \cite[Theorem 35]{MR0118704}). Here, 
$H=\left(
\begin{array}{cc}
 0 & \frac{1}{2} \\
 \frac{1}{2} & 0 \\
\end{array}
\right)$ and $Y=\left(
\begin{array}{cc}
 1 & \frac{1}{2} \\
 \frac{1}{2} & 1 \\
\end{array}
\right)$.
We say that a matrix of the form (\ref{eqn:Watson}) is \textit{Watson-reduced}. 

First note that a matrix $(b_{ij})\in \mat{2}{2}$ of type II with $\ord(b_{12})=0$ and $d=b_{11}b_{22}-b_{12}^2$ is $GL_2(\ZZ_2)$-equivalent to either $2H$ or $2Y$ if $d\equiv 7 \pmod 8$ or $d\equiv 3 \pmod 8$, respectively. Thus, we can assume that $B$ is of the form
\begin{equation}\label{eqn:preWatson}
2^{e_1'}(V_1'\bot U_1') \bot \ldots \bot 2^{e_m'}(V_m' \bot U_m')
\end{equation}
as in (\ref{eqn:Watson}), but $U_i' =u_1 \bot \cdots \bot u_{k_i}$ not necessarily with $0\leq k_i\leq 2$, and
\[
V_i' =H\bot \dots \bot H \bot \underbrace{Y\bot \dots \bot Y}_{m_i}
\]
with $m_i\geq 2$ possibly.

To make a reduction into the form (\ref{eqn:Watson}), we can use the following facts (see, for example, \cite[118p.]{MR522835}) : $Y\bot Y\sim H\bot H$, and for any $u_1,u_2,u_3\in \{1,3,5,7\}$, there exist $u'\in \{1,3,5,7\}$ and $K\in \{H,Y\}$ such that 
\begin{equation}\label{eqn:uuutou2k}
(u_1)\bot (u_2)\bot (u_3)\sim (u')\bot 2K.
\end{equation}
See Table \ref{table:uuutou2K} for an explicit description of (\ref{eqn:uuutou2k}). Applying these rules repeatedly, we get (\ref{eqn:Watson}) equivalent to (\ref{eqn:preWatson}).

\begin{tiny}
\begin{table}[ht]
\begin{tabular}{ccc|ccc}
$(u_1,u_2,u_3)$ & $(u',K)$ & $\text{genus symbol}$ & $(u_1,u_2,u_3)$ & $(u',K)$ & $\text{genus symbol}$ \\
\hline
$(1, 1, 1)$	&	$(3, Y)$	&	$\subsuperscript{1}{3}{3}$	&	$(3, 3, 3)$	&	$(1, Y)$	&	$\subsuperscript{1}{1}{-3}$	\\
$(1, 1, 3)$	&	$(5, H)$	&	$\subsuperscript{1}{5}{-3}$	&	$(3, 3, 5)$	&	$(3, H)$	&	$\subsuperscript{1}{3}{-3}$	\\
$(1, 1, 5)$	&	$(7, Y)$	&	$\subsuperscript{1}{7}{-3}$	&	$(3, 3, 7)$	&	$(5, Y)$	&	$\subsuperscript{1}{5}{3}$	\\
$(1, 1, 7)$	&	$(1, H)$	&	$\subsuperscript{1}{1}{3}$	&	$(3, 5, 5)$	&	$(5, H)$	&	$\subsuperscript{1}{5}{-3}$	\\
$(1, 3, 3)$	&	$(7, H)$	&	$\subsuperscript{1}{7}{3}$	&	$(3, 5, 7)$	&	$(7, H)$	&	$\subsuperscript{1}{7}{3}$	\\
$(1, 3, 5)$	&	$(1, H)$	&	$\subsuperscript{1}{1}{3}$	&	$(3, 7, 7)$	&	$(1, Y)$	&	$\subsuperscript{1}{1}{-3}$	\\
$(1, 3, 7)$	&	$(3, H)$	&	$\subsuperscript{1}{3}{-3}$	&	$(5, 5, 5)$	&	$(7, Y)$	&	$\subsuperscript{1}{7}{-3}$	\\
$(1, 5, 5)$	&	$(3, Y)$	&	$\subsuperscript{1}{3}{3}$	&	$(5, 5, 7)$	&	$(1, H)$	&	$\subsuperscript{1}{1}{3}$	\\
$(1, 5, 7)$	&	$(5, H)$	&	$\subsuperscript{1}{5}{-3}$	&	$(5, 7, 7)$	&	$(3, H)$	&	$\subsuperscript{1}{3}{-3}$	\\
$(1, 7, 7)$	&	$(7, H)$	&	$\subsuperscript{1}{7}{3}$	&	$(7, 7, 7)$	&	$(5, Y)$	&	$\subsuperscript{1}{5}{3}$	\\
\end{tabular}
\caption{$(u_1,u_2,u_3)$ and $(u',K)$ such that $(u_1)\bot (u_2)\bot (u_3)\sim (u')\bot 2K$}
\label{table:uuutou2K}
\end{table}
\end{tiny}

% \begin{example}
% Use the command \verb+reduceWatson+ for this procedure.

% \begin{mmaCell}[moredefined={reduceWatson, JC, Y}]{Input}
%   reduceWatson[JC[\{0,\{1,1,1,1,1\}\},\{1,\{Y\}\}]]
% \end{mmaCell}

% \begin{mmaCell}{Output}
%   JC[\{0,\{5\}\},\{1,\{H,H,H\}\}]
% \end{mmaCell}

% \end{example}
%latex table align margin

{\bf Weak canonical reduction.}
Let $B\in \mat{n}{2}$. We want to find a weak canonical form equivalent to $B$. Assume that $B$ is Watson-reduced, that is, it is of the form (\ref{eqn:Watson})
\[
2^{e_1}(V_1\bot U_1) \bot \ldots \bot 2^{e_m}(V_m \bot U_m),
\]
where each summand $2^{e_j}(V_j\bot U_j)$ is given as either
\[
 \underbrace{(2^{e_j}H\bot \dots \bot 2^{e_j}H)}_{\text{possibly empty}} \bot (2^{e_j}U_j),\,
 \text{or }  \underbrace{(2^{e_j}H\bot \dots \bot 2^{e_j}H)}_{\text{possibly empty}} \bot (2^{e_j}Y) \bot (2^{e_j}U_j)
\]
with possibly empty $U_j$. By enumerating all non-empty Jordan components above for $1\leq j \leq m$, without disturbing the order, we obtain the decomposition 
\[
B=2^{k_1}J_1\bot\cdots\bot 2^{k_r}J_r,
\]
where $J_i$ is one of $H$, $Y$, and $U_j$ for some $1\leq j\leq m$. Such a decomposition clearly satisfies Conditions (1) and (2) of Definition \ref{def.2.1}; in plain language, these conditions mean that (1) the list $(2^{k_i}J_i)_{1\leq i \leq r}$ of $B$ are sorted by the size of $k_i$ so that $k_1$ is the smallest, (2) a diagonal matrix $2^{k_i}J_i$ comes after the components of the form $2^{k}K$ with $K \in \{H, Y\}$ with $k=k_i$, and there is no other diagonal matrix $2^{k}J$ with $k=k_i$.

Now we claim that the submatrix of $B$ obtained by removing all components of the form $2^{k_i}J_i$ with $J_i \in \{H, Y\}$ satisfies Condition (3) of Definition \ref{def.2.1} if and only if so does $B$. To see this, note that $\deg_{B''\bot 2^k K} \equiv \deg_{B''} \pmod 2$ and $\ordet{(B''\bot 2^k K)} \equiv \ordet{B''} \pmod 2$ for any $B''\in \mat{n}{2}$ and $k\in \ZZ_{\geq 0}$, where $K$ is either $H$ or $Y$. Moreover, when the degree of $B''$ is even, $\xi_{B''\bot 2^k K}= 0\iff \xi_{B''}= 0$.

Therefore, we can simply assume that $B = 2^{k_1}J_1\bot\cdots\bot 2^{k_r}J_r$ with $k_1<\dots <k_r$, where each $J_i$ is of the form $(u_1)$ or $(u_1) \bot (u_2)$ for some $u_1,u_2 \in \{1,3,5,7\}$, and that it is not a weak canonical form but satisfies (1) and (2) of Definition \ref{def.2.1}. So let $j\in \{1,\dots, r-1\}$ be the minimal position at which Condition (3) is not met, namely, $k_{j+1}=k_{j}+1$, and $\deg B^{[j]}$ and $\ordet{B^{[j]}}$ are even, but $\xi_{B^{[j]}}\neq 0$. Since $\deg B^{[j]}$ is even, $B^{[j]}$ has at least two diagonal entries, and thus, $B^{[j]}$ and $B$ take the following forms :
\begin{equation}\label{eqn:wcf1}
B^{[j]} = 
\left(\begin{array}{@{}c|c@{}}
B_j& 0 \\
\hline
  0 &
2^{k_j}u_1  
\end{array}\right)
,\qquad
B=\left(\begin{array}{@{}c|c@{}}
B^{[j]} & 0 \\
\hline
  0 &
  \begin{matrix}
  2^{k_j+1}u_2 & 0 \\
  0 & \ddots
  \end{matrix}
\end{array}\right).
\end{equation}
Here, $u_{1}$ is the last entry of $J_{j}$, $u_{2}$ is the first entry of $J_{j+1}$, and $B_j$ is the diagonal submatrix of $B^{[j]}$ obtained by removing its last diagonal $2^{k_j}u_1$. So $\deg B_j \geq 1$.

Recall that $\xi_{B^{[j]}} = \chi_2((-1)^{\deg B^{[j]}/2} \det B^{[j]})$, where
for $a = 2^{r} c\in \ZZ_2^{\times}$ with $\ord(c) =0$, 
$$\chi_2(a) =
\begin{cases}
1 & \text{if $r \equiv 0 \bmod 2, \, c \equiv 1 \bmod 8$} \\
-1 & \text{if $r \equiv 0 \bmod 2, \, c \equiv 5 \bmod 8$} \\ 
0 & \text{otherwise}
\end{cases}
$$
(cf. \cite[428p.]{Kat}). Thus, $\xi_{B^{[j]}} = \chi_2((-1)^{\deg B^{[j]}/2} 2^{\ordet{B^{[j]}}}u_0u_1)$, where $\det B_j=2^k u_0$ with $k=\ord(B_j)\in \ZZ_{\geq 0}$. As $\ordet{B^{[j]}} = k+k_j$ is even, the condition $\xi_{B^{[j]}}\neq 0$ implies that $(-1)^{\deg B^{[j]}/2}u_0u_1\equiv 1\,\text{or }5 \pmod 8$. Choose $c\in \{1,3,5,7\}$ such that $c\equiv (-1)^{\deg B^{[j]}/2}u_0 \pmod 8$. Then $cu_1\equiv 1 \pmod 4$.

Let us describe a rule to replace only these $u_1,u_2$ from $B$ into $u_1',u_2'$ to get a new matrix $B'$ such that $\xi_{B^{'[j]}}= 0$ and
\begin{equation}\label{eqn:wcf2}
B=\left(\begin{array}{@{}c|c@{}}
B_j & 0 \\
\hline
  0 &
  \begin{matrix}
2^{k_j}u_1 &  0 & 0 \\
0 &  2^{k_j+1}u_2 & 0 \\
0 &  0 & \ddots
  \end{matrix}
\end{array}\right)
\sim
\left(\begin{array}{@{}c|c@{}}
B_j & 0 \\
\hline
  0 &
  \begin{matrix}
2^{k_j}u_1' &  0 & 0 \\
0 &  2^{k_j+1}u_2' & 0 \\
0 &  0 & \ddots
  \end{matrix}
\end{array}\right)
=B'.
\end{equation}
Note that if such an operation exists, then $B'$ still satisfies Conditions (1) and (2) of Definition \ref{def.2.1}, and possibly violates Condition (3) only at positions strictly bigger than $j$. 
%$B^{'[i]}=B^{[i]}$ for all $i< j$
Therefore, applying such a replacement rule ultimately leads to a weak canonical form equivalent to $B$.

\begin{lemma}\label{lma:uutouu1}
Let $c,u_1,u_2\in \{1,3,5,7\}$ be as above. There exist $u_1',u_2'\in \{1,3,5,7\}$ satisfying the following :
\begin{itemize}
\item If $u_1\not\equiv u_2 \pmod 4$, then $u_i\not\equiv u_i' \pmod 4$, and $\left(\frac{u_i}{2}\right)=\left(\frac{u_i'}{2}\right)$ for each $i=1,2$, and $u_1+u_2\equiv u_1'+u_2' \pmod 8$;
\item If $u_1\equiv u_2 \pmod 4$, then $u_i\not\equiv u_i' \pmod 4$, and $\left(\frac{u_i}{2}\right)\neq \left(\frac{u_i'}{2}\right)$ for each $i\in \{1,2\}$, and $u_1+u_2+4\equiv u_1'+u_2' \pmod 8$.
\end{itemize}
We have $cu_1' \equiv 3\ \pmod 4$, and $(2^{k_j}u_1)\bot (2^{k_j+1}u_2)\sim (2^{k_j}u_1')\bot (2^{k_j+1}u_2')$.
\end{lemma}

\begin{proof}
See Table \ref{table:uutouu} for the existence of such $u_1'$ and $u_2'$. 
As $cu_1 \equiv 1\ \pmod 4$, $cu_1' \equiv 3\ \pmod 4$. Regarding the last part of the statement, it is enough to show that $(u_1)\bot (2 u_2)\sim (u_1')\bot (2u_2')$. Let us borrow a few terms from \cite[Section 7.5 of Chapter 15]{MR1194619} for a simple proof. When $u_1\not\equiv u_2 \pmod 4$, the genus symbols of $(u_1)\bot (2 u_2)$ and $(u_1')\bot (2u_2')$ are the same up to the `oddity fusion'. When $u_1\equiv u_2 \pmod 4$, the genus symbols of $(u_1)\bot (2 u_2)$ and $(u_1')\bot (2u_2')$ are the same up to the `sign walk'.
\end{proof}

\begin{tiny}
\begin{table}[ht]
\begin{tabular}{c|cc|cc}
& $(u_1,u_2)$ &  $(u_1',u_2')$ & $(u_1,u_2)$ &  $(u_1',u_2')$ \\
\hline
\multirow{2}{*}{$u_1\not\equiv u_2 \pmod 4$} 
& $(1, 3)$  &  	$(7, 5)$  &  	$(5, 3)$  &  	$(3, 5)$ \\
& $(1, 7)$  &  	$(7, 1)$  &  	$(5, 7)$  &  	$(3, 1)$ \\
& $(3, 1)$  &  	$(5, 7)$  &  	$(7, 1)$  &  	$(1, 7)$ \\
& $(3, 5)$  &  	$(5, 3)$  &  	$(7, 5)$  &  	$(1, 3)$ \\
\hline
\multirow{2}{*}{$u_1\equiv u_2 \pmod 4$}
& $(1, 1)$  &  	$(3, 3)$  &  	$(5, 1)$  &  	$(7, 3)$ \\
& $(1, 5)$  &  	$(3, 7)$  &  	$(5, 5)$  &  	$(7, 7)$ \\
& $(3, 3)$  &  	$(1, 1)$  &  	$(7, 3)$  &  	$(5, 1)$ \\
& $(3, 7)$  &  	$(1, 5)$  &  	$(7, 7)$  &  	$(5, 5)$
\end{tabular}
\caption{$(u_1,u_2,u_1',u_2')$ satisfying the conditions in Lemma \ref{lma:uutouu1}.}
\label{table:uutouu}
\end{table}
\end{tiny}

%The command \verb+reduceWeakCanonical+ follows the above procedure.

\appendix
\section{Tables of intersection numbers of three modular correspondences}
In this appendix we provide tables of arithmetic intersection numbers of three modular correspondences generated by \textit{computeGK}.
As mentioned in the Introduction, the Gross-Keating invariants for ternary quadratic forms over $\Zp$ have been originally introduced to express those numbers. Let us denote the set of non-degenerate half-integral matrices with entries in $\ZZ$ by $\Zmat{n}$. We can regard $Q\in \Zmat{n}$ as an element of $\mat{n}{p}$ for any prime $p$.

For $m\in \ZZ_{\geq 1}$, let $\phi_m(X,Y)\in \ZZ[X,Y]$ be the classical modular polynomial; see \cite{GK} and the references therein. Let $m_1,m_2,m_3\in \ZZ_{\geq 1}$. Gross and Keating showed that the cardinality of the quotient ring $\ZZ[X,Y]/(\phi_{m_1},\phi_{m_2},\phi_{m_3})$ is finite if and only if there is no positive definite binary quadratic form $a x^2+bxy+cy^2$ with $a,b,c\in \ZZ$ which represents the three integers $m_1,m_2,m_3$. Assume that $m_1,m_2,m_3$ satisfy this condition. Let $S=\mathrm{Spec}\, \ZZ[X,Y]$ and $T_m$ be the divisor on $S$ corresponding to $\phi_m$. We define the arithmetic intersection number as follows :
\begin{equation}\label{eqn:TTT}
\begin{aligned}
\intmult : & = \log \# \ZZ[X,Y]/(\phi_{m_1},\phi_{m_2},\phi_{m_3}) \\
& = \sum_{p}n(p)\log p,
\end{aligned}
\end{equation}
with $n(p)=0$ for $p>4m_1m_2m_3$. Gross and Keating found an explicit formula for $n(p)$.
\begin{theorem}\cite[Proposition 3.22]{GK}\label{thm:GKformula}
Let $p$ be a prime. We have
$$
n(p) = \frac{1}{2}\sum_{Q}\left(\prod_{l\mid \Delta,\, l\neq p} \beta_l(Q) \right)\cdot \alpha_p(Q),
$$
with $\Delta = 4\det Q\in \ZZ$. The sum is over all positive definite matrices $Q\in \Zmat{3}$ with diagonal $(m_1,m_2,m_3)$ which are isotropic over $\QQ_{l}$ for all $l\neq p$. Such $Q$ is anisotropic over $\Qp$ and $p$ divides $\Delta$. The quantities $\alpha_p(Q)$ and $\beta_p(Q)$ are given as follows.

Let $H = (a_1, a_2, a_3; \vep_1, \vep_2, \vep_3)$ be a naive EGK datum of $Q$ regarded as elements of
$\mat{3}{p}$. When $a_1\equiv a_2 \pmod 2$ and $a_2<a_3$, we further define $\epsilon$ to be $\vep_2$.
If $a_1\equiv a_2 \pmod 2$, then $\alpha_p(Q)$ is equal to
\begin{multline*}
\sum_{i=0}^{a_1-1} (i+1) (a_1+a_2+a_3-3 i)p^i \\ +\sum_{i=a_1}^{(a_1+a_2-2)/2} (a_1+1) (2a_1+a_2+a_3-4 i) p^i+\frac{1}{2} (a_1+1) (a_3-a_2+1) p^{(a_1+a_2)/2}.
\end{multline*}
If $a_1\not\equiv a_2 \pmod 2$, then $\alpha_p(Q)$ is equal to
$$
\sum_{i=0}^{a_1-1} (i+1) (a_1+a_2+a_3-3 i)p^i +\sum _{i=a_1}^{(a_1+a_2-1)/2} (a_1+1) (2a_1+a_2+a_3-4i)p^i.
$$
If $a_1\equiv a_2 \pmod 2$ and either $\epsilon =1$ or $a_2=a_3$, then $\beta_p(Q)$ is equal to
$$
\sum _{i=0}^{a_1-1} 2(i+1)p^i +\sum _{i=a_1}^{(a_1+a_2-2)/2} 2(a_1+1)p^i+(a_1+1) (a_3-a_2+1) p^{(a_1+a_2)/2}.
$$
If $a_1\equiv a_2 \pmod 2$ and $\epsilon =-1$, then $\beta_p(Q)$ is equal to
$$
\sum _{i=0}^{a_1-1} 2(i+1)p^i +\sum _{i=a_1}^{(a_1+a_2-2)/2} 2(a_1+1)p^i+(a_1+1) p^{(a_1+a_2)/2}.
$$
If $a_1\not\equiv a_2 \pmod 2$, then $\beta_p(Q)$ is equal to
$$
\sum _{i=0}^{a_1-1} 2(i+1)p^i +\sum _{i=a_1}^{(a_1+a_2-2)/2} 2(a_1+1)p^i.
$$
\end{theorem}
By using this formula with the help of \textit{computeGK}, we have produced tables of $n(p)$ when $1\leq m_1\leq m_2 \leq m_3\leq 30$; see Tables \ref{table:int1} and \ref{table:int2}. The necessary steps are as follows :
\begin{enumerate}
\item[(1)] Find triples $\vecm=(m_1,m_2,m_3)$ such that there is no positive definite binary integral quadratic form representing these integers;
\item[(2)] For such a triple $\vecm$ and a prime $p\leq 4m_1m_2m_3$, find the list 
$\LBp$ of all positive definite matrices $Q\in \Zmat{3}$ with diagonal $(m_1,m_2,m_3)$, which are isotropic over $\QQ_{l}$ for all $l\neq p$;
\item[(3)] For $Q\in \LBp$, compute $\alpha_p(Q)$, and $\beta_l(Q)$ for primes $l\mid 4\det Q, \,l\neq p$, and
$$
n(p) = \frac{1}{2}\sum_{Q\in \LBp}\left(\prod_{l} \beta_l(Q) \right)\cdot \alpha_p(Q).
$$
\end{enumerate}
The last step is about computing a naive EGK datum of $Q$ over $\Zp$, which is covered in Section \ref{sec:negk}. We focus on the first two steps here.

\begin{tiny}
\begin{table}[ht]
\begin{adjustbox}{center}{
\begin{tabular}{c|cccccccccccccc}
 $(m_1,m_2,m_3)$ & 2 & 3 & 5 & 7 & 11 & 13 & 17 & 19 & 23 & 29 & 31 & 37 & 41 & 43 \\
\hline
 (1,2,15) & 450 & 108 & 36 & 156 & 28 & 72 & 36 & 36 & 18 & 10 & 34 & 8 & 22 & 8 \\
 (1,2,21) & 656 & 144 & 298 & 32 & 40 & 96 & 48 & 68 & 24 & 12 & 42 & 10 & 8 & 8 \\
 (1,2,30) & 738 & 324 & 108 & 468 & 88 & 216 & 108 & 144 & 84 & 32 & 104 & 24 & 58 & 16 \\
 (1,3,10) & 452 & 182 & 36 & 144 & 58 & 72 & 60 & 24 & 54 & 52 & 14 & 12 & 10 & 4 \\
 (1,3,14) & 604 & 144 & 290 & 24 & 104 & 96 & 60 & 32 & 60 & 84 & 22 & 16 & 28 & 20 \\
 (1,3,26) & 1060 & 252 & 508 & 336 & 132 & 12 & 132 & 56 & 81 & 136 & 32 & 28 & 68 & 28 \\
 (1,3,29) & 2164 & 180 & 366 & 240 & 108 & 120 & 108 & 44 & 54 & 8 & 32 & 12 & 20 & 12 \\
 (1,3,30) & 1812 & 290 & 144 & 576 & 266 & 288 & 216 & 96 & 186 & 202 & 56 & 48 & 92 & 32 \\
 (1,5,22) & 1484 & 1456 & 108 & 434 & 10 & 216 & 108 & 216 & 72 & 64 & 122 & 58 & 30 & 88 \\
 (1,6,13) & 1152 & 432 & 512 & 336 & 304 & 12 & 108 & 80 & 69 & 32 & 32 & 50 & 20 & 26 \\
 (1,7,26) & 2124 & 2216 & 1014 & 84 & 248 & 24 & 404 & 200 & 216 & 180 & 58 & 92 & 216 & 72 \\
 (1,7,30) & 3624 & 1516 & 288 & 182 & 532 & 576 & 648 & 340 & 444 & 268 & 106 & 120 & 354 & 48 \\
 (1,8,15) & 882 & 540 & 180 & 762 & 414 & 360 & 324 & 180 & 114 & 86 & 134 & 150 & 58 & 156 \\
 (1,8,21) & 1336 & 720 & 1498 & 144 & 612 & 480 & 432 & 340 & 186 & 100 & 154 & 170 & 36 & 116 \\
 (1,8,30) & 1314 & 1620 & 540 & 2286 & 1332 & 1080 & 972 & 720 & 678 & 296 & 398 & 300 & 158 & 208 \\
 (1,10,21) & 3952 & 1488 & 396 & 144 & 652 & 576 & 480 & 408 & 276 & 204 & 198 & 252 & 64 & 190 \\
 (1,10,23) & 2740 & 2912 & 216 & 864 & 442 & 440 & 428 & 288 & 42 & 172 & 170 & 232 & 64 & 138 \\
 (1,10,27) & 4540 & 362 & 360 & 1448 & 862 & 744 & 600 & 528 & 522 & 360 & 148 & 188 & 84 & 300 \\
 (1,11,30) & 5476 & 1296 & 576 & 1892 & 178 & 864 & 648 & 724 & 264 & 470 & 420 & 216 & 346 & 168 \\
 (1,12,29) & 4504 & 1260 & 2586 & 1698 & 942 & 856 & 756 & 776 & 427 & 52 & 380 & 192 & 180 & 168 \\
 (1,13,24) & 2360 & 2180 & 2580 & 1716 & 1296 & 60 & 540 & 400 & 780 & 260 & 332 & 160 & 212 & 228 \\
 (1,13,30) & 6908 & 2752 & 504 & 2040 & 1604 & 82 & 664 & 480 & 906 & 400 & 320 & 156 & 300 & 160 \\
 (1,14,27) & 6080 & 288 & 2956 & 240 & 1556 & 972 & 600 & 704 & 612 & 580 & 212 & 224 & 296 & 340 \\
 (1,15,26) & 6340 & 1512 & 680 & 2036 & 982 & 72 & 796 & 504 & 306 & 488 & 320 & 464 & 460 & 616 \\
 (1,17,30) & 9148 & 1944 & 648 & 2612 & 1424 & 1316 & 122 & 724 & 1092 & 464 & 632 & 492 & 292 & 360 \\
 (1,19,30) & 9100 & 3624 & 984 & 2924 & 1780 & 1442 & 880 & 106 & 708 & 1080 & 530 & 444 & 592 & 192 \\
 (2,4,15) & 1026 & 756 & 252 & 1070 & 500 & 504 & 396 & 360 & 284 & 118 & 200 & 124 & 94 & 84 \\
 (2,4,21) & 1424 & 1008 & 2092 & 208 & 632 & 672 & 528 & 440 & 340 & 164 & 312 & 178 & 64 & 172 \\
 (2,4,30) & 3114 & 2268 & 756 & 3220 & 1448 & 1512 & 1188 & 1008 & 684 & 328 & 640 & 432 & 258 & 400 \\
 (2,6,23) & 2964 & 2252 & 2736 & 1750 & 944 & 864 & 612 & 576 & 10 & 420 & 320 & 176 & 224 & 140 \\
 (2,10,21) & 6256 & 2592 & 864 & 432 & 2188 & 1728 & 1248 & 1160 & 776 & 760 & 746 & 468 & 460 & 484 \\
 (2,13,30) & 10372 & 4536 & 2044 & 6120 & 3000 & 216 & 2392 & 2016 & 1428 & 1490 & 1008 & 1224 & 1276 & 1488 \\
 (3,5,22) & 5480 & 1296 & 610 & 1742 & 104 & 864 & 650 & 432 & 280 & 340 & 292 & 270 & 270 & 452 \\
 (3,7,26) & 9316 & 2016 & 4188 & 434 & 1866 & 96 & 1032 & 952 & 480 & 472 & 480 & 440 & 340 & 276 \\
 (3,7,30) & 15828 & 2352 & 1588 & 576 & 2900 & 2308 & 1728 & 1632 & 924 & 764 & 792 & 924 & 516 & 648 \\
 (5,7,26) & 12700 & 13594 & 1008 & 646 & 2386 & 168 & 1256 & 1276 & 1392 & 756 & 734 & 672 & 560 & 340 \\
 (5,7,30) & 21828 & 5184 & 688 & 864 & 3976 & 3464 & 2160 & 2178 & 1976 & 1412 & 1262 & 1284 & 798 & 872 \\
\end{tabular}}
\end{adjustbox}
\caption{$n(p)$ for given $(m_1,m_2,m_3)$ and $2\leq p\leq 43$}
\label{table:int1}
\end{table}

\begin{table}[ht]
\begin{adjustbox}{center}{
\begin{tabular}{c|ccccccccccccccccc}
 $(m_1,m_2,m_3)$ & 47 & 53 & 59 & 61 & 67 & 71 & 73 & 79 & 83 & 89 & 97 & 101 & 103 & 107 & 109 & 113 & 127 \\
\hline
 (1,2,15) & 26 & 2 & 12 & 14 & 0 & 1 & 12 & 2 & 8 & 10 & 8 & 8 & 6 & 0 & 2 & 0 & 0 \\
 (1,2,21) & 33 & 6 & 36 & 14 & 8 & 10 & 12 & 10 & 18 & 8 & 4 & 10 & 12 & 0 & 2 & 0 & 2 \\
 (1,2,30) & 94 & 18 & 44 & 50 & 8 & 15 & 32 & 16 & 36 & 30 & 24 & 26 & 30 & 8 & 6 & 0 & 2 \\
 (1,3,10) & 10 & 10 & 2 & 6 & 0 & 9 & 4 & 6 & 10 & 2 & 0 & 8 & 0 & 6 & 2 & 2 & 0 \\
 (1,3,14) & 15 & 14 & 2 & 16 & 0 & 6 & 10 & 20 & 4 & 2 & 0 & 10 & 6 & 6 & 16 & 0 & 8 \\
 (1,3,26) & 56 & 38 & 28 & 42 & 0 & 50 & 16 & 28 & 14 & 2 & 0 & 24 & 4 & 4 & 20 & 2 & 18 \\
 (1,3,29) & 16 & 16 & 29 & 28 & 0 & 22 & 20 & 24 & 16 & 6 & 0 & 12 & 8 & 8 & 8 & 2 & 4 \\
 (1,3,30) & 74 & 70 & 36 & 52 & 0 & 57 & 20 & 36 & 54 & 20 & 0 & 40 & 12 & 20 & 20 & 8 & 4 \\
 (1,5,22) & 44 & 42 & 60 & 12 & 16 & 58 & 22 & 41 & 40 & 8 & 12 & 8 & 8 & 24 & 0 & 16 & 14 \\
 (1,6,13) & 16 & 32 & 32 & 16 & 12 & 42 & 12 & 22 & 12 & 22 & 12 & 12 & 12 & 8 & 4 & 6 & 0 \\
 (1,7,26) & 176 & 74 & 108 & 52 & 56 & 116 & 16 & 16 & 38 & 72 & 12 & 14 & 15 & 4 & 48 & 20 & 4 \\
 (1,7,30) & 248 & 136 & 184 & 74 & 76 & 190 & 26 & 36 & 130 & 126 & 48 & 46 & 38 & 44 & 56 & 56 & 4 \\
 (1,8,15) & 78 & 58 & 32 & 126 & 36 & 10 & 68 & 46 & 68 & 76 & 32 & 32 & 32 & 8 & 26 & 12 & 6 \\
 (1,8,21) & 103 & 164 & 96 & 96 & 44 & 38 & 72 & 74 & 102 & 80 & 36 & 34 & 42 & 24 & 18 & 20 & 12 \\
 (1,8,30) & 268 & 256 & 148 & 272 & 112 & 72 & 148 & 210 & 268 & 196 & 136 & 82 & 162 & 76 & 74 & 40 & 34 \\
 (1,10,21) & 92 & 144 & 96 & 54 & 124 & 98 & 142 & 62 & 66 & 48 & 82 & 42 & 24 & 32 & 32 & 50 & 20 \\
 (1,10,23) & 56 & 106 & 32 & 32 & 190 & 54 & 52 & 69 & 38 & 66 & 38 & 34 & 32 & 46 & 60 & 36 & 20 \\
 (1,10,27) & 80 & 138 & 60 & 124 & 176 & 83 & 140 & 134 & 120 & 80 & 64 & 94 & 80 & 80 & 68 & 46 & 44 \\
 (1,11,30) & 124 & 136 & 172 & 292 & 116 & 138 & 180 & 138 & 152 & 110 & 112 & 148 & 68 & 118 & 80 & 92 & 46 \\
 (1,12,29) & 96 & 108 & 66 & 92 & 128 & 66 & 184 & 80 & 52 & 36 & 48 & 58 & 54 & 96 & 80 & 100 & 52 \\
 (1,13,24) & 264 & 196 & 288 & 56 & 80 & 240 & 76 & 120 & 124 & 96 & 52 & 28 & 28 & 40 & 38 & 42 & 32 \\
 (1,13,30) & 488 & 240 & 384 & 72 & 84 & 266 & 100 & 146 & 308 & 182 & 62 & 96 & 88 & 132 & 50 & 72 & 26 \\
 (1,14,27) & 147 & 212 & 38 & 208 & 192 & 38 & 244 & 216 & 104 & 116 & 134 & 94 & 112 & 106 & 148 & 98 & 60 \\
 (1,15,26) & 344 & 162 & 152 & 230 & 236 & 58 & 256 & 100 & 156 & 162 & 78 & 94 & 114 & 64 & 212 & 60 & 90 \\
 (1,17,30) & 510 & 172 & 308 & 262 & 238 & 244 & 206 & 276 & 216 & 126 & 196 & 128 & 254 & 112 & 194 & 116 & 42 \\
 (1,19,30) & 424 & 448 & 466 & 240 & 182 & 415 & 184 & 234 & 248 & 266 & 232 & 252 & 146 & 118 & 146 & 122 & 82 \\
 (2,4,15) & 162 & 102 & 76 & 130 & 38 & 29 & 72 & 70 & 108 & 90 & 62 & 44 & 72 & 32 & 26 & 10 & 10 \\
 (2,4,21) & 137 & 102 & 156 & 182 & 80 & 62 & 88 & 142 & 158 & 88 & 50 & 70 & 134 & 38 & 44 & 30 & 46 \\
 (2,4,30) & 462 & 282 & 208 & 442 & 118 & 89 & 228 & 196 & 308 & 290 & 162 & 150 & 218 & 62 & 84 & 42 & 68 \\
 (2,6,23) & 210 & 132 & 60 & 128 & 96 & 96 & 142 & 94 & 70 & 86 & 52 & 144 & 64 & 128 & 92 & 102 & 102 \\
 (2,10,21) & 352 & 372 & 564 & 322 & 340 & 354 & 156 & 466 & 214 & 92 & 184 & 76 & 250 & 202 & 216 & 174 & 268 \\
 (2,13,30) & 1328 & 746 & 588 & 650 & 634 & 380 & 652 & 426 & 544 & 498 & 272 & 234 & 426 & 200 & 522 & 278 & 320 \\
 (3,5,22) & 252 & 138 & 114 & 128 & 178 & 56 & 144 & 68 & 92 & 220 & 90 & 124 & 62 & 118 & 68 & 48 & 44 \\
 (3,7,26) & 348 & 362 & 450 & 222 & 192 & 404 & 256 & 116 & 124 & 132 & 224 & 172 & 64 & 52 & 60 & 80 & 60 \\
 (3,7,30) & 572 & 712 & 650 & 304 & 408 & 488 & 464 & 210 & 362 & 272 & 352 & 216 & 150 & 126 & 100 & 266 & 92 \\
 (5,7,26) & 312 & 516 & 258 & 554 & 434 & 164 & 344 & 344 & 148 & 336 & 270 & 216 & 290 & 84 & 144 & 78 & 216 \\
 (5,7,30) & 520 & 896 & 666 & 802 & 746 & 374 & 736 & 718 & 468 & 500 & 526 & 384 & 498 & 246 & 272 & 248 & 316 \\
\end{tabular}}
\end{adjustbox}
\caption{$n(p)$ for given $(m_1,m_2,m_3)$ and $47\leq p\leq 127$}
\label{table:int2}
\end{table}
\end{tiny}

{\bf First step.}
The first step of the above can be reformulated as follows.

\begin{proposition}\cite[Proposition 3.5]{MR2340368}\label{prop:gortz}
Let $m_1, m_2, m_3$ be positive integers. The following are equivalent :
\begin{itemize}
\item There exists no positive definite integral binary quadratic form which represents $m_1, m_2$, and $m_3$.
\item Every positive semi-definite half-integral symmetric matrix $Q$ with diagonal $(m_1, m_2, m_3)$ is non-degenerate.
\end{itemize}
\end{proposition}
While the first condition is difficult to use for the purpose of checking if $\vecm=(m_1,m_2,m_3)$ satisfies it, the second condition is suitable for algorithmic treatment. The fist condition is still useful to rule out a given $\vecm$ quickly by tabulating the first few coefficients of the theta functions 
\[
\theta_{a,b,c}=\sum_{(x,y)\in \ZZ^2} q^{ax^2+bxy+cy^2}
\]
for a reasonably big finite set of positive definite binary quadratic forms $ax^2+bxy+cy^2$ of small $4ac-b^2$. 

Let $Q$ be a half-integral symmetric matrix with diagonal $(m_1, m_2, m_3)$, that is, 
%For a given $\vecm=(m_1,m_2,m_3)$, $Q\in \Zmat{3}$ 
%It takes the following form :
\[
Q = \left(
\begin{array}{ccc}
 m_1 & \frac{t_3}{2} & \frac{t_2}{2} \\
 \frac{t_3}{2} & m_2 & \frac{t_1}{2} \\
 \frac{t_2}{2} & \frac{t_1}{2} & m_3 \\
\end{array}
\right)
\]
for some $t_i\in \ZZ$. For $Q$ to be positive semi-definite, $(t_1,t_2,t_3)\in \ZZ^3$ should satisfy the following inequalities for the principal minors :
\begin{equation}\label{eqn:tminor}
4 m_1 m_2-t_3^2\geq 0,\qquad t_1t_2 t_3-m_1 t_1^2-m_2 t_2^2-m_3 t_3^2+4 m_1 m_2 m_3\geq 0.
\end{equation}
When all such $(t_1,t_2,t_3)$, which form a finite set, satisfy the strict inequalities in the above, then Theorem \ref{thm:GKformula} applies to $(m_1,m_2,m_3)$. There are 37 such triples $(m_1,m_2,m_3)$ with $1\leq m_1\leq m_2\leq m_3\leq 30$.

Before we turn to the next step, we note that Theorem \ref{thm:GKformula} applies to infinitely many $(m_1,m_2,m_3)$. In other words, there are infinitely many triples satisfying the first condition of Proposition \ref{prop:gortz}. We can even obtain an infinite list with $m_1=1$ and $m_2=2$. If a positive definite integral binary quadratic form represents both $1$ and $2$, then it must be equivalent to one of $x^2+y^2$, $x^2+2y^2$, $x^2+xy+2y^2$. But there are infinitely many integers not represented by any of these three forms. For example, any prime $p$ such that $p \equiv 7 \pmod 8$ and $p$ is a quadratic non-residue modulo $7$ cannot be represented any of them.

\begin{tiny}
\begin{table}
\begin{adjustbox}{center}{
\begin{tabular}{c|ccccccc}
$\vecm$ & $\#\LA$ & $\#\LBB$ & $\#\LB{2}$ & $\#\LB{3}$ & $\#\LB{5}$ & $\#\LB{7}$ & $\#\LB{11}$\\
\hline
 (1,2,15) & 1119 & 1107 & 266 & 90 & 33 & 148 & 32 \\
 (1,2,21) & 1583 & 1567 & 374 & 134 & 232 & 27 &  48 \\
 (1,2,30) & 2251 & 2239 & 266 & 198 & 65 & 320 & 80 \\
 (1,3,10) & 1195 & 1192 & 288 & 122 & 40 & 128 & 72 \\
 (1,3,14) & 1669 & 1647 & 349 & 122 & 238 & 22 & 100 \\
 (1,3,26) & 3141 & 3061 & 611 & 200 & 374 & 280 & 134 \\
 (1,3,29) & 3487 & 3293 & 1266 & 204 & 332 & 236 & 132 \\
 (1,3,30) & 3539 & 3504 & 778 & 122 & 92 & 352 & 204 \\
 (1,5,22) & 4387 & 4262 & 760 & 956 & 66 & 356 & 10 \\
 (1,6,13) & 3017 & 2966 & 600 & 220 & 392 & 280 & 248 \\
 (1,7,26) & 7263 & 6923 & 1118 & 1396 & 696 & 50 & 244 \\
 (1,7,30) & 8251 & 8130 & 1530 & 626 & 178 & 88 & 376 \\
 (1,8,15) & 4663 & 4631 & 266 & 352 & 121 & 508 & 264 \\
 (1,8,21) & 6611 & 6579 & 374 & 502 & 900 & 101 & 408 \\
 (1,8,30) & 9315 & 9261 & 266 & 716 & 241 & 996 & 576 \\
 (1,10,21) & 8359 & 8229 & 1626 & 627 & 220 & 98 & 472 \\
 (1,10,23) & 9195 & 8689 & 1404 & 1788 & 171 & 628 & 340 \\
 (1,10,27) & 10777 & 10644 & 2020 & 122 & 258 & 848 & 560 \\
 (1,11,30) & 12901 & 12563 & 2225 & 780 & 290 & 990 & 62 \\
 (1,12,29) & 13483 & 13093 & 1266 & 880 & 1512 & 1084 & 676 \\
 (1,13,24) & 12323 & 12154 & 600 & 906 & 1554 & 1120 & 808 \\
 (1,13,30) & 15333 & 14930 & 2648 & 1090 & 306 & 1192 & 872 \\
 (1,14,27) & 15015 & 14647 & 2521 & 122 & 1710 & 150 & 800 \\
 (1,15,26) & 15361 & 14841 & 2626 & 904 & 306 & 1116 & 618 \\
 (1,17,30) & 20067 & 19319 & 3394 & 1188 & 375 & 1476 & 784 \\
 (1,19,30) & 22199 & 21364 & 3580 & 1366 & 440 & 1496 & 1024 \\
 (2,4,15) & 4665 & 4647 & 266 & 368 & 127 & 532 & 292 \\
 (2,4,21) & 6555 & 6519 & 374 & 518 & 930 & 83 & 380 \\
 (2,4,30) & 9391 & 9347 & 266 & 732 & 241 & 1044 & 568 \\
 (2,6,23) & 10729 & 10426 & 928 & 768 & 1346 & 880 & 504 \\
 (2,10,21) & 16193 & 15995 & 1608 & 1162 & 333 & 162 & 950 \\
 (2,13,30) & 30917 & 30253 & 2626 & 1960 & 662 & 2368 & 1376 \\
 (3,5,22) & 12939 & 12537 & 2302 & 806 & 284 & 1000 & 55 \\
 (3,7,26) & 21541 & 20728 & 3604 & 1198 & 2220 & 238 & 1060 \\
 (3,7,30) & 24715 & 24325 & 4642 & 627 & 548 & 284 & 1368 \\
 (5,7,26) & 35591 & 33382 & 5234 & 6692 & 584 & 298 & 1424 \\
 (5,7,30) & 41147 & 39853 & 6999 & 2492 & 214 & 374 & 1816 \\
\end{tabular}}
\end{adjustbox}
\caption{The sizes of $\LA$, $\LBB$, and $\LB{p}$}
\label{table:sizesofL}
\end{table}
\end{tiny}

{\bf Second step.}
Now we have a triple $\vecm = (m_1,m_2,m_3)$ and the list $\LA$ of positive definite matrix $Q\in \Zmat{3}$ with diagonal $(m_1,m_2,m_3)$ satisfying (\ref{eqn:tminor}). For each $Q\in \LA$, let us find all primes $p$ such that $Q$ is anisotropic over $\Qp$.
\begin{lemma}\cite[Lemma 3.1.7]{MR522835}\label{lma:cas}
Let $p$ be odd. Assume that $Q$ is $GL_3(\Qp)$-equivalent to $a_1x_1^2+a_2x_2^2+a_3x_3^2$. If $\ord_{p}(a_1) = \ord_{p}(a_2) = \ord_{p}(a_3)$, then $Q$ is isotropic.
\end{lemma}
To obtain a finite list of primes $p$ such that $Q$ is anisotropic over $\Qp$, first find a diagonal matrix $Q'$ with diagonals $a_1,a_2,a_3\in \ZZ$ such that $Q$ and $Q'$ are $GL_3(\QQ)$-equivalent. Then $Q$ is isotropic over $\Qp$ if $p \nmid 2a_1a_2a_3$ by Lemma \ref{lma:cas}. Thus, we only need to check whether $Q$ is anisotropic over $\Qp$, which is equivalent to the condition $(-a_1a_3,-a_2a_3)=-1$ involving the Hilbert symbol over $\Qp$, for primes $p \mid 2a_1a_2a_3$.

In this way, we obtain a subset $\LBB$ of $\LA$ consisting of $Q\in \LA$ such that $Q$ is anisotropic over $\Qp$ for some unique prime $p$, and isotropic over $\QQ_l$ for all $l\neq p$. For each prime $p$, let
$$
\LBp:=\{Q\in \LBB: \text{$Q$ is anisotropic over $\Qp$}\}.
$$
Note that we only need $\LBp$ for $p\leq 4m_1m_2m_3$. See Table \ref{table:sizesofL} for the sizes of $\LA$, $\LBB$ and $\LBp$ for some primes $p$.

{\bf Reliability checks.}
When $m_1=1$, we can calculate the resultant of $\phi_{m_2}(X,X)$ and $\phi_{m_3}(X,X)$ to obtain $\prod_{p}p^{n(p)}$. It requires an explicit expression\footnote{The data for $\phi_m$ is available at \url{https://math.mit.edu/~drew/ClassicalModPolys.html}.} for $\phi_m(X,Y)$; see, for example, \cite{MR2869057,MR3435725} for a method to compute $\phi_m$. We have checked the entries with $m_1=1$ in Tables \ref{table:int1} and \ref{table:int2} by calculating the resultants of these polynomials, and obtained the agreement.
\begin{example}
There are values of $n(p)$ for $(m_1,m_2,m_3)=(1,3,10)$ in \cite{GK}, which are said to be computed by Elkies using the resultant of the polynomials $\phi_{3}(X,X)$ and $\phi_{10}(X,X)$; they seem to be the unique example that had appeared in the literature. Table \ref{table:int1} shows that $n(2)=452$ for $(m_1,m_2,m_3)=(1,3,10)$, which matches Elkies' computation.
% \begin{mmaCell}[moredefined={computeGKint}]{Input}
%   computeGKint[\{1,3,10\},2]
% \end{mmaCell}
% \begin{mmaCell}{Output}
%   452
% \end{mmaCell}
\end{example}

%%%%%%%%%%%%%%%%%%%%%%%%%%%%%%%%%%%%%%%%%%%%


\begin{thebibliography}{99}

\bibitem{Bouw}
{I.~I.~Bouw},
\emph{Invariants of ternary quadratic forms},
Ast\'erisque  \textbf{312}  (2007) 
121--145.

\bibitem{MR2869057}
{R.~Br\"{o}ker, K.~Lauter, and A.~V. Sutherland}, \emph{Modular polynomials via
  isogeny volcanoes}, Math. Comp. \textbf{81} (2012), 
  1201--1231.

\bibitem{MR3435725}
{J.~H. Bruinier, K.~Ono, and A.~V. Sutherland}, \emph{Class polynomials for
  nonholomorphic modular functions}, J. Number Theory \textbf{161} (2016),
  204--229.

\bibitem{MR522835}
J.~W.~S. Cassels, \emph{Rational quadratic forms}, Academic Press, Inc.,
1978.



\bibitem{Cho}
{S.~Cho}, 
\emph{Group schemes and local densities of quadratic lattices in residue characteristic 2}, Compositio Math. \textbf{151} (2015) 793--827.  

\bibitem{Cho1}
{S.~Cho}, 
\emph{On the local density formula and the Gross-Keating invariant (with an appendix by Tamotsu Ikeda and Hidenori Katsurada)}, arXiv:1805.03331. to appear in Math. Z. 


\bibitem{CY}
{S.~Cho and T. Yamauchi}
\emph{A reformulation of the Siegel series and intersection numbers}, arXiv:1805.01666. to appear in Math. Annalen.


\bibitem{MR1194619}
J.~H. Conway and N.~J.~A. Sloane, \emph{Sphere packings, lattices and groups},
  second ed., Springer-Verlag,
  1993.

\bibitem{MR2340368}
U.~G\"{o}rtz, \emph{Arithmetic intersection numbers},
Ast\'erisque  \textbf{312}  (2007)
15--24


\bibitem{ARGOS}
{U.~G\"ortz and M.~Rapoport, ed.},
\emph{ARGOS seminar on Intersections of Modular Correspondences},
Ast\'erisque \textbf{312}  Soci\'et\'e Math\'ematique de France  (2007).




\bibitem{GK}
{B.~Gross and K.~Keating},
\emph{On the intersection of modular correspondences},
Inv. Math.
\textbf{112}
(1993)

\bibitem{Ha}
{G. Harder}, \emph{A congruence between a Siegel nodular form and an elliptic modular form}, manuscript 225--245., 2003, reproduced in: J. H. Bruinier, G. van der Geer, G. Harder, D. Zagier (Eds), The 1-2-3 of Modular Forms , Springer Verlag, Berlin, Heidelberg, 2008, 147--160.

\bibitem{IK1}
{T.~Ikeda and H.~Katsurada},
\emph{On the Gross-Keating invariants of a quadratic forms over a non-archimedean local field},  
Amer. J. Math. 
\textbf{140}
(2018)
1521-1565


\bibitem{IK2}
{T.~Ikeda and H.~Katsurada},
\emph{An explicit formula for the Siegel series of a half-integral matrix over the ring of integers in  a non-archimedean local field}, 
preprint, http://arxiv.org/abs/1602.06617.

\bibitem{Kat}
{H.~Katsurada},
\emph{An explicit formula for Siegel series},
Amer. J. Math.
\textbf{121}
(1999)
415--452.

\bibitem{Kat2}
{H.  Katsurada}, \emph{Congruence for the Klingen Eisenstein series and Harder's conjecture},
Report of Oberwolfach Workshop `Modular forms' organized by J. Bruinier, A. Ichino, T. Ikeda and O. Imamoglu, 57(2019), 55--57.

\bibitem{MR1954971}
O.~D. King, \emph{A mass formula for unimodular lattices with no roots}, Math.
  Comp. \textbf{72} (2003), 
  839--863.

\bibitem{MR3739221}
O.~D. King, C.~Poor, J.~Shurman, and D.~S. Yuen, \emph{Using {K}atsurada's
  determination of the {E}isenstein series to compute {S}iegel eigenforms},
  Math. Comp. \textbf{87} (2018), 
  879--892.



\bibitem{k1}
{S.~Kudla},
 \emph{Central derivatives of Eisenstein series and height pairings}, 
 Ann of math. (2) \textbf{146}, no.3 (1997) 545-646.

\bibitem{kry}
{S.~Kudla, M.~Rapoport, and T.~Yang},
\emph{Modular Forms and Special Cycles on Shimura Curves},
Princeton university press,
(2006).


\bibitem{lz} 
{C.~Li and W.~Zhang},
\emph{Kudla-Rapoport cycles and derivatives of local densities}, 
arXiv:1908.01701.


\bibitem{omeara}
{O.~T.~O'Meara},
\emph{Introduction to Quadratic Forms},
Springer, (1973).



\bibitem{scharlau}
{W.~Scharlau},
\emph{Quadratic and Hermitian forms}
Springer (1985)



\bibitem{shimura97}
{G.~Shimura},
\emph{Euler products and Eisenstein series},
AMS, (1997).


\bibitem{thyang}
{T.~Yang},
\emph{Local densities of 2-adic quadratic forms},
J.~Number Theory
\textbf{108}
(2004)
287--345.

\bibitem{Wat}
{G.~L. Watson},
\emph{The $2$-adic densities of a quadratic form},  Mathematika \textbf{23} (1976) 94--106.


\bibitem{MR0118704}
G.~L. Watson, \emph{Integral quadratic forms}, Cambridge University Press, 
1960.

\bibitem{Wedhorn}
{T.~Wedhorn},
\emph{Calculation of representation densities},
Ast\'erisque  \textbf{312}  (2007) 185--196.


\end{thebibliography}
\end{document}